\documentclass[12pt,twoside,reqno]{amsart}
\usepackage{mathptmx, amsmath, amssymb, amsfonts, amsthm, mathptmx, enumerate, color, mathrsfs}
\usepackage[colorlinks=false,citecolor=blue]{hyperref}
\usepackage{graphicx}
\graphicspath{ C:\Users\pqdee\OneDrive\Desktop\bf}
\graphicspath{ C:\Users\pqdee\OneDrive\Desktop}
\usepackage{caption}
\usepackage{subcaption}
\usepackage[export]{adjustbox}
\usepackage{float}
\usepackage[utf8]{inputenc}
\usepackage[T1]{fontenc}
\usepackage{diagbox}
\usepackage{makecell}
\newcommand{\vertiii}[1]{{\left\vert\kern-0.25ex\left\vert\kern-0.25ex\left\vert #1 
    \right\vert\kern-0.25ex\right\vert\kern-0.25ex\right\vert}}
\hypersetup{colorlinks=true,linkcolor=red, anchorcolor=green, citecolor=cyan, urlcolor=red, filecolor=magenta, pdftoolbar=true}
\usepackage{multicol}

\newcommand{\C}{\mathbb{C}}
\newcommand{\D}{\mathbb{D}}

\newcommand{\g}{\gamma}
\renewcommand{\l}{\lambda}
\usepackage{epstopdf}
\newtheorem{theorem}{Theorem}[section]
\newtheorem{lemma}{Lemma}[section]
\newtheorem{remark}{Remark}[section]
\newtheorem{definition}{Definition}[section]
\newtheorem{corollary}{Corollary}[section]
\newtheorem{example}{Example}[section]

\numberwithin{equation}{section}
\newtheorem{case}{Case}

\allowdisplaybreaks

\begin{document}
\title{Convexity of the Berezin range of operators on $\mathcal{H}_\gamma (\mathbb{D})$ }
\author[S. Maiti] {Sandip Kumar Maiti}
\author[S. Sahoo] {Satyajit Sahoo}
\author[ G. Chakraborty] {Gorachand Chakraborty}

\address{(Maiti) Department of Mathematics, Panchakot Mahavidyalaya, Purulia, West Bengal 723121, India}
\email{pq.deep@gmail.com}

\address{(Sahoo) Department of Mathematics, School of Basic Sciences, Indian Institute of Technology Bhubaneswar, Bhubaneswar, Odisha 752050, India}
\email{satyajitsahoo2010@gmail.com, ssahoomath@gmail.com}

\address{(Chakraborty) Department of Mathematics, Sidho-Kanho-Birsha University, Purulia, West Bengal 723104, India}
\email{gorachand.chakraborty@skbu.ac.in }

 \thanks{ Dr. Satyajit Sahoo is thankful to  IIT Bhubaneswar for providing the necessary facilities to carry out this work.}

\subjclass[2020]{47B32, 52A10}
\keywords{Reproducing kernel; Berezin transform; Berezin range;
Convexity; Finite-rank operator; Hardy space; Bergman space.}

\begin{abstract}
In this paper, we characterize the convexity of the Berezin range for finite-rank operators acting on the weighted Hardy space $\mathcal{H}_\gamma (\mathbb{D})$ over the unit disc $\mathbb{D}$. We provide a complete classification in terms of convexity for concrete operators. Additionally, we address dynamical properties of finite-rank operators on Hardy and Bergman spaces. Several illustrative examples are discussed to support our theoretical findings. Additionally,
geometrical interpretations have also been employed.

\end{abstract}
\maketitle
\pagestyle{myheadings}
\markboth{\centerline{}}
{\centerline{}}
\bigskip
\bigskip
\section{Introduction}

\noindent 
The study of operator ranges and their geometric properties has been a central topic in functional analysis. In particular, the Berezin range, which associates an operator with a subset of the complex plane via Berezin transform techniques, provides insight into spectral and structural properties. 
This work focuses on the convexity of the Berezin range for \emph{finite-rank operators} acting on \( \mathcal{H}_\gamma(\mathbb{D}) \), the weighted Hardy space over the unit disc. We establish a complete classification of convexity for such operators. 

Given a non-empty set $\Omega$, let $\mathcal{F}(\Omega,\mathbb{C})$ be the collection of all functions from $\Omega$ to the field $\mathbb{C}$. We note that $\mathcal{F}(\Omega,\mathbb{C})$ is a vector space over $\mathbb{C}$. A subset $\mathcal{H}$ of $\mathcal{F}(\Omega,\mathbb{C})$ is called a \textit{reproducing kernel Hilbert space} (RKHS) if (i) $\mathcal{H}$ is a subspace of $\mathcal{F}(\Omega,\mathbb{C})$, (ii) $\mathcal{H}$ is a Hilbert space with respect to an inner product $\langle\cdot,\cdot\rangle$, and (iii) for all $\lambda\in \Omega$, the linear evaluation functional $E_\lambda:\mathcal{H}\rightarrow \mathbb{C},\;h\mapsto h(\lambda)$ is bounded. 
 Let $\mathcal{H}$ be a RKHS over a set $\Omega$. Then by Riesz representation theorem, for each evaluation functional $E_\lambda$, as defined above, there exists a unique vector $k_\lambda\in \mathcal{H}$ such that $E_\lambda(h)=\langle h,k_\lambda\rangle=h(\lambda).$ This function $k_\lambda$ is called the \textit{reproducing kernel function} (or kernel function as a shorthand) at the point $\lambda\in \Omega$. For more details on RKHS we refer \cite {Paulsen_Book_2016}.\par
The concept of the numerical range of an operator goes back to O. Toeplitz, who defined in 1918 the field of values of a matrix, a concept easily extensible to bounded linear operators on a Hilbert space. Let $\mathcal{L}(\mathcal{H})$ denotes the collection of all bounded linear operators on a RKHS $\mathcal{H}$. The \textit{numerical range} of an operator $T\in\mathcal{L(H)}$ is defined to be the complex set
 \begin{equation*}
     W(T)=\left\{\left \langle Tx,x\right \rangle:x\in\mathcal{H}, \|x\|=1\right\},
 \end{equation*}
where $\left\langle \cdot,\cdot\right\rangle$ is the inner product defined on $\mathcal{H}$ and $\|\cdot\|$ is the induced norm on $\mathcal{H}$.
It is well known that $W(T)$ is a bounded convex set, which is called the \textit{Toeplitz-Hausdorff} Theorem \cite{Gau_Pei_2021, Gustafson_PAMS_1970}, but the exact shape of it is uncertain, in general.
The radius of the smallest circular disc with center at the origin containing $W(T)$ is defined to be the numerical radius $w(T)$ of $T\in\mathcal{L(H)}.$ Thus, if $T\in\mathcal{L(H)}$, then 
$w(T)=\sup\{|z|:z\in W(T)\}.$ \par
For a bounded linear operator $T$ acting on a RKHS $\mathcal{H}$ over a set $\Omega$, the \textit{Berezin transform} (or \textit{Berezin symbol}) of $T$ at a point $\lambda\in \Omega$ is defined as
\begin{equation*}
    \widetilde{T}(\lambda)=\langle T\widehat{k}_\lambda, \widehat{k}_\lambda\rangle_\mathcal{H},
\end{equation*}
where $\widehat{k}_\lambda=\dfrac{k_\lambda}{\|k_\lambda\|}$ is the normalized reproducing kernels. Also, the \textit{Berezin range} (or \textit{Berezin set}) of $T$ is defined as
\begin{equation*}
    \mbox{Ber}(T)=\Bigl \{\widetilde{T}(\lambda):\lambda\in \Omega \Bigr \}=\Bigl \{\langle T\widehat{k}_\lambda,\widehat{k}_\lambda\rangle\}_\mathcal{H}:\lambda\in \Omega\Bigr \},
\end{equation*}
and the Berezin radius (or the Berezin number) of $T$ is
\begin{equation*}
     \mbox{ber}(T)=\sup_{\lambda\in \Omega}|\widetilde{T}(\lambda)|=\sup \{|z|:z\in  \mbox{Ber}(T)\}.
\end{equation*}
Clearly, the Berezin symbol $\widetilde{T}$ is a bounded function on $\Omega$ whose values lie in the numerical range of the operator $T$, and hence $\mbox{Ber}(T)\subseteq W(T) ~\mbox{and}~  \mbox{ber}(T)\leq w(T).$  Also, it is obvious that (see \cite{Karaev_Article_2006}, \cite{Karaev2_Article_2006}) the Berezin symbol $\widetilde{T}$ is a bounded function on $\Omega$ and $\sup_{\lambda\in \Omega}|\widetilde{T}(\lambda)|$, which is the Berezin number of $T$ does not exceed $\|T\|,$ i. e., $ \mbox{ber}(T)=\sup_{\lambda\in \Omega}|\widetilde{T}(\lambda)|\leq \|T\|.$

A well-known example of a reproducing kernel Hilbert space is the classical Hardy Hilbert space $H^2$ \cite{Paulsen_Book_2016} on the unit disk $\mathbb{D}$,
$$H^2(\mathbb{D})=\left\{f(z)=\sum_{n\geq 0} a_nz^n\in \mbox{Hol}(\mathbb{D}): \sum_{n\geq 0}|a_n|^2<\infty\right\},$$

where $\mbox{Hol}(\mathbb{D})$ is the collection of holomorphic functions on $\mathbb{D}$. If $f(z)=\sum_{n\geq 0} a_nz^n$ and $g(z)=\sum_{n\geq 0} b_nz^n$  are elements of $H^2$, then their inner product is given by $\langle f, g\rangle=\sum_{n\geq 0} a_n\overline{b_n}$. 
 The reproducing kernel
for $H^2$, known as the Szego kernel, is given by $k_\lambda(z)=\frac{1}{1-\overline{\lambda}z}$, $z,\lambda\in \mathbb{D}$. However, another RKHS to keep in mind is the Bergman space \cite{Paulsen_Book_2016} on the unit disc $\mathbb{D}$

$$A^2(\mathbb{D})=\left\{f\in \mbox{Hol}(\mathbb{D}):\int_{\mathbb{D}}|f(z)|^2dV(z)<\infty\right\},$$
where $dV$ is the normalized volume measure on $\mathbb{D}$. The reproducing kernel
for $A^2(\mathbb{D})$, is given by $k_\lambda(z)=\frac{1}{(1-\overline{\lambda}z)^2}$, $z, \lambda \in \mathbb{D}$. 
Let $\mathcal{H}_\gamma (\mathbb{D})$ be the Hilbert space of holomorphic functions on $\mathbb{D}$ with the reproducing kernel $$ k_\lambda(z)=(1-\overline{\lambda}z)^{-\gamma},$$ for all $\lambda, z\in \C$, and $\gamma>0$. Note that for $\gamma=1~ (\mbox{resp}. ~ \gamma=2)$, $\mathcal{H}_\gamma$ reduces to the classical  Hardy space (resp. Bergman space).  In fact, such a space was studied in \cite{Le_JMAA_2012} for several complex variables and for $\gamma>0$.
In this definition,
\begin{equation*}
    \frac{1}{\left(1-\overline{\lambda}z\right)^\gamma}=\exp \left(\gamma\; \mbox{Log}\frac{1}{1-\overline{\lambda}z}\right),
\end{equation*}
where $\mbox{Log}$ denotes the principal branch of logarithm (in general, we define $a^\gamma=\exp\left(\gamma \;\mbox{Log}a\right)$ for every $a\in \mathbb{C} \setminus \left\{0\right\}$).
Clearly, $\mathcal{H}_\gamma(\D)$ is the completion of the linear span of $\left\{k_\lambda:\lambda\in \mathbb{D}\right\}$ with respect to the inner product defined by $\langle k_\lambda,k_z\rangle=k_\lambda(z)$ with 
\begin{equation}
    \|k_\lambda\|^2=\langle k_\lambda,k_\lambda\rangle=k_\lambda(\lambda)=\frac{1}{(1-|\lambda|^2)^\gamma}.
\end{equation}
The normalized reproducing kernel at the point $\lambda$, denoted by $\widehat{k}_\lambda,$ is given by
\begin{equation}
    \widehat{k}_{\lambda}(z)=\dfrac{k_\lambda(z)}{\|k_\lambda\|}=\dfrac{(1-|\lambda|^2)^{\frac{\gamma}{2}}}{(1-\overline{\lambda}z)^\gamma}.
\end{equation}
For more details on weighted Hardy space, we refer the readers to \cite [Section 2.1]{Cowen_MacCluer_1995} and \cite{Zhu_Article_2008}.

We now provide some expository remarks to motivate our results and highlight the diverse applications of the Berezin transform.

\section{Background and Motivation}

\noindent 

The Berezin set and Berezin number, were allegedly first formally introduced by Karaev in \cite{Karaev_Article_2006} while F. Berezin himself in \cite{Berezin_Article_1972} first introduced the notion of Berezin symbol. This notion has been proven to be a critical tool in operator theory, as many significant properties of important operators are encrypted in their Berezin transforms. In particular, it is known that on the most familiar RKHS including the Hardy space, the Bergman space, the Fock space, and the Dirichlet space, the Berezin transform uniquely determines the operator, i.e., if $T_1, T_2\in \mathcal{L(H)}$,  then $T_1=T_2\;\Leftrightarrow \widetilde{T}_1=\widetilde{T}_2.$ \par
The very first important result involving the Berezin symbol involves the invertibility of Toeplitz operators acting on $H^2.$ R. G. Douglas in \cite{Douglas_Conf_1972} asked the question: if $\varphi \in L^{\infty}(\mathbb{T})$ with $|\widetilde{T}_\varphi|\geq \delta>0,$ is the Toeplitz operator invertible? Tolokonnikov in \cite{Tolo_Art_1981} and then Wolff in \cite{Wolff_Article_2002} answered to this question in affirmative provided $\delta$ is sufficiently closed to $1.$ \par

 Let $\varphi$ be a holomorphic self-map of $\mathbb{D}$, then the equation $C_\varphi(f)=f\circ \varphi$ defines a composition operator $C_\varphi$ with the inducing map $\varphi$. 
These operators are appreciated by many and have a significant and deep roots in operator and function theory (see \cite{Cowen_MacCluer_1995, Shapiro_1993}). Berger and Coburn \cite{Berger_Coburn_1986} asked: {\it if the Berezin symbol of an operator on the Hardy or Bergman space vanishes on the boundary of the disk, must the operator be compact? }
Nordgren and Rosenthal \cite{Nordgren_Rosenthal_1992} have demonstrated, on a so-called standard reproducing kernel Hilbert space, that if the Berezin symbols of all unitary equivalents of an operator vanish on the boundary, then the operator is compact. The counterexamples presented come in the form of composition operators.
Another significant finding from Axler and Zheng \cite{Axler_Zheng_1998}, is that if $S$ is a finite sum of finite products of Toeplitz operators acting on the Bergman space of the unit disk, then $S$ is compact if and only if the Berezin symbol of $S$ vanishes as it approaches the boundary of the disk. 
One motivation for studying the Berezin range of these operators is that they often elude  Axler-Zheng type results; e.g. there are composition operators such that $\widetilde{C}_\varphi(z) \rightarrow 0$ as $ z \rightarrow \partial \mathbb{D}$, but $C_\varphi$ is not compact.

Another important field of investigation involving Berezin symbol is the geometry of Berezin range. Over the years, numerous mathematicians have extensively explored the geometry of the numerical range and developed a wide array of numerical radius inequalities. So, it is quite natural to ask whether the results involving numerical range and numerical radius of an operator $T$ are valid or have a similar extension for the Berezin range and the Berezin radius of $T.$ The convexity of Berezin range has been investigated in several works including \cite{CowenFelder}, \cite{Augustine_Garayev_Shankar_CAOT_2023}. Some extensive studies on Berezin radius inequalities can be found in \cite{Bhunia_Article_2023, Bhunia2_Article_2023, Garayev_Article_2023, Tap_Article_2021, Zamani_Article_2024}. Recently, Augustine \textit{et al.} in \cite{Athul2_Article_2024} investigated the convexity of finite-rank operators on the Hardy space and on the Bergman space and also developed some Berezin radius inequalities. In this article, we investigate the convexity of finite-rank operators on the weighted Hardy space $\mathcal{H}_\g (\D).$

\subsection{Convexity of the Berezin range}

The numerical range of an operator plays a crucial role in operator theory, matrix analysis, and it exhibits several notable properties. For instance, it is well known that the spectrum of an operator lies within the closure of its numerical range. Moreover, the numerical range is always convex, a result known as the Toeplitz-Hausdorff Theorem. For additional background on the numerical range, we refer the reader to \cite{GRO, halmos}. While considerable effort has been devoted to characterizing the geometry of the numerical range (see, for example, \cite{Daepp_Gorkin_Shaffer_Voss_2018, Kippenhahn_2008}), there appear to be only a few results concerning the geometry of the Berezin range \cite[Section 2.1]{Karaev}, and, to the best of the authors’ knowledge, none of these address the question of convexity.\\ \\
As the convexity of the numerical range stands out as one of its most distinctive property, it serves as the motivation for the primary question addressed in this work.
\\ \\
{\bf Question-$1$:} {\bf Given a bounded operator acting on reproducing kernel Hilbert space, is it true that the Berezin range of the operator is convex? Conversely, if Berezin range of the operator is convex, what can be said about the operator? }\\ 

This question was first raised by Karaev \cite{Karaev}, and we provide answers for a few classes of specific operators.

By \textit{Toeplitz-Hausdorff} Theorem, the numerical range of an operator is always convex \cite{Gustafson_PAMS_1970}. It is not very difficult to observe that the Berezin range of an operator is always a subset of the numerical range. In general, the Berezin range of an operator need not be convex. What fails in the Toeplitz-Hausdorff theorem when restricting to normalized reproducing kernels? Karaev \cite{Karaev} has initiated the study of the geometry of the Berezin range. In fact Karaev \cite{Karaev} has proved that the Berezin range of the Model operator $M_{z^n}$ on the model space is convex. Motivated by this, recently Cowen and Felder \cite{CowenFelder} explored the convexity of the Berezin range for matrices, multiplication operators on reproducing kernel Hilbert space, and composition operators acting on Hardy space of the unit disc.
In light of the Cowen-Felder's results, Augustine {\it et al.} \cite{Augustine_Garayev_Shankar_CAOT_2023} have characterized the convexity of the Berezin range of composition operators acting on Bergman space over the unit disk.  
The Berezin transform is most useful on any reproducing kernel Hilbert space (RKHS), and the Berezin range may naturally be more relevant in this context. Now, we propose the following problem.
\\ 

{\bf Question-$2$: What can be said about the convexity of the Berezin range of a special class of operators acting on $\mathcal{H}_\gamma(\mathbb{D})$?}
\\

The geometry of the Berezin range of a matrix is notably simple compared to the numerical range of the matrix. Additionally, we note that the trace of a matrix can be recovered as the sum of the elements in its Berezin range. More broadly, it is known in certain spaces (see, for instance, \cite[Proposition 3.3]{ZK}) that if $T$
 is a trace-class (or positive) operator, its trace can be recovered via the Berezin transform. Since the Berezin transform is straightforward in the finite-dimensional setting, we now turn our attention to infinite-dimensional spaces.

\subsection{Multiplication operators}
Let us recall that for any Hilbert (or Banach) space of functions $\mathcal{H}$,
$$\textnormal{Mult}(\mathcal{H}):=\{g\in \mathcal{H}: gf \in \mathcal{H}~~\textnormal{for all}~ f\in \mathcal{H}\}.$$
For $g\in \textnormal{Mult}(\mathcal{H})$, the multiplication operator $M_g$ on $\mathcal{H}$ by $M_gf=gf$.
\begin{lemma}\cite[Proposition 3.2]{CowenFelder}
    Let $\mathcal{H}$ be an RKHS on a set $X$ and $g\in \textnormal{Mult}(\mathcal{H})$. Then the Berezin range of $M_g$ is convex if and only if $g(X)$ is convex.
\end{lemma}
At first glance, one might expect the convexity of the Berezin range to be easily understood. However, this is not generally the case. To illustrate this point, we examine certain classes of operators acting on $\mathcal{H}_\gamma$, where the characterization of convexity becomes significantly more technically involved.

The aim of the present paper is to analyze and answer the above {\bf Question-2} for the finite-rank operator.

Karaev \cite{Karaev} investigated the Berezin range of the rank one operator of the form $T(f)=\langle f, z\rangle z$ acting on Hardy space $H^2(\mathbb{D})$. Motivated by this Augustine {\it et al.} \cite{Athul2_Article_2024} further explored the convexity of finite-rank operators on the Hardy space, on the Bergman space and also developed some Berezin radius inequalities very recently.
 Inspired by this, we investigate the convexity for the Berezin range of the finite-rank operator on $\mathcal{H}_\gamma(\mathbb{D})$.
The main ingredient of this paper is to characterize the convexity of the Berezin range of certain concrete operators. More specifically, we characterize the Berezin range of finite-tank operators acting on $\mathcal{H}_\gamma(\mathbb{D})$.
The present article narrows down to the precise question of investigating the convexity of the Berezin range of finite-tank operators acting on $\mathcal{H}_\gamma(\mathbb{D})$.
\\ 

Among the results presented in \cite{Athul2_Article_2024}, the main obstacle to obtaining analogous results in the reproducing kernel Hilbert space $\mathcal{H}_\gamma(\mathbb{D})$ is the greater complexity of the reproducing kernel, which is given by $ k_\lambda(z)=(1-\overline{\lambda}z)^{-\gamma}.$


\section{Results and discussions}

In this section, we characterize the convexity of the Berezin range for finite-rank operators acting on $\mathcal{H}_\gamma(\mathbb{D})$. Further, we geometrically illustrate the region of convexity for several sets of parameters using Python.
\subsection{Berezin range of finite rank operator on the weighted Hardy space}

\noindent 

An operator acting on a Hilbert space $\mathcal{H}$ is said to be of finite rank if $\mbox{dim}(R(T))$ is finite. In particular, if $\mbox{dim}(R(T))=1,$ then $T$ is called a \textit{rank-one operator}. It is well known that if $T$ be a bounded linear operator on a Hilbert space $\mathcal{H}$, then $Tf=\langle f,\phi \rangle \psi$ for all $f\in \mathcal{H},$ where $\psi$ is a nonzero vector in $R(T)$ and $\phi$ is a fixed unique element in $\mathcal{H}.$\par
First, we deal with the simple most rank-one operator of the form $T(f)=\langle f,z \rangle z$ acting on $\mathcal{H}_\gamma(\mathbb{D}).$

\begin{theorem}\label{th:(f,z)z}
    Let $T$ be a rank-one operator of the form $T(f)=\langle f,z\rangle z$ acting on $\mathcal{H}_\gamma(\mathbb{D}).$ Then, $\textnormal{Ber}(T)=\left[0,\;\dfrac{\gamma^\gamma}{\left(1+\gamma\right)^{1+\gamma}}\right],$ which is convex in $\mathbb{C}.$
\end{theorem}
    \begin{proof}
        For $\lambda\in \mathbb{D}$, we have
        \begin{equation*}
            T(k_\lambda)=\langle k_\lambda,z\rangle z=\overline{\langle z,k_\lambda\rangle}z=\overline{\lambda}z,
        \end{equation*}
        so that the Berezin transform of $T$ at $\lambda$ is
        \begin{equation*}
            \widetilde{T}(\lambda)=\langle T\widehat{k}_\lambda,\widehat{k}_\lambda\rangle=\dfrac{1}{\|k_\lambda\|^2}\langle Tk_\lambda,k_\lambda\rangle=\left(1-|\lambda|^2\right)^\gamma (Tk_\lambda)(\lambda)=\left(1-|\lambda|^2\right)^\gamma |\lambda|^2,
        \end{equation*}
        where $|\lambda|\in[0,1).$ We observe that $\widetilde{T}(\lambda)=\widetilde{T}(|\lambda|)$ is a real analytic function. Now we search for the extreme points of $\widetilde{T}(\lambda)$. We differentiate $\widetilde{T}(\lambda)$ w.r.to $|\lambda|$ to have
        \begin{equation*}
            \dfrac{d}{d|\lambda|}\left[\widetilde{T}(\lambda)\right]=2|\lambda|\left(1-|\lambda|^2\right)^{\gamma-1}\left[1-(1+\gamma)|\lambda|^2\right].
        \end{equation*}
        Clearly, $ \dfrac{d}{d|\lambda|}\left[\widetilde{T}(\lambda)\right]=0,\;\text{when}\; |\lambda|^2=0,\;\dfrac{1}{1+\gamma}$.\\
        Now, $\left.\widetilde{T}(\lambda)\right \vert_{|\lambda|^2=0}=0.$ Also, $\left.\widetilde{T}(\lambda)\right \vert_{|\lambda|^2=\frac{1}{1+\gamma}}=\dfrac{\gamma^\gamma}{(1+\gamma)^{1+\gamma}}.$\\
        Thus, $\left.\widetilde{T}(\lambda)\right\vert_{\text{min}}=0$ and $\left.\widetilde{T}(\lambda)\right\vert_{\text{max}}=\dfrac{\gamma^\gamma}{(1+\gamma)^{(1+\gamma)}}.$ Since $T$ is bounded, $\widetilde{T}(\lambda)$ is a continuous real analytic function. Hence by the Intermediate value theorem, $\widetilde{T}(\lambda)$ takes on any real value between $0$ and $\dfrac{\gamma^\gamma}{\left(1+\gamma\right)^{1+\gamma}}.$ Hence, $\textnormal{Ber}(T)=\left[0,\;\dfrac{\gamma^\gamma}{\left(1+\gamma\right)^{1+\gamma}}\right],$ which is a convex set in $\mathbb{C}.$
    \end{proof}
    \begin{remark}
        On putting $\gamma=1$ and $\gamma=2$ respectively in Theorem \ref{th:(f,z)z}, we see that the Berezin ranges of the rank-one operator $T$ given by $T(f)=\langle f,z\rangle z$ acting on the Hardy space $H^2(\mathbb{D})$ and the Bergman space $A^2(\mathbb{D})$ are respectively $\left[0,\;\frac{1}{4}\right]$ and $\left[0,\;\frac{4}{27}\right]$ (see Figure \ref{fig:theorem_1,(f,z)z,g=1,2,0.01,5_wc}), both of which are convex sets in $\mathbb{C}$, which have been proved in \cite{Athul2_Article_2024}.
    \end{remark}
    \begin{remark}\label{re:th_1}
        Using calculus, it can be shown that for $\gamma>0,$ the function $\dfrac{\gamma^\gamma}{\left(1+\gamma\right)^{1+\gamma}}$ is monotone decreasing with $\lim_{\gamma \rightarrow 0^+}\dfrac{\gamma^\gamma}{\left(1+\gamma\right)^{1+\gamma}}=1$ and $\lim_{\gamma \rightarrow \infty}\dfrac{\gamma^\gamma}{\left(1+\gamma\right)^{1+\gamma}}=0.$ This indicates that the Berezin range of the operator $T(f)=\langle f,z \rangle$ acting on $\mathcal{H}_\gamma(\mathbb{D})$ squeezes with increasing $\gamma$ and ultimately for very large value of $\gamma$ it tends to coincide with the origin, which can also be verified with the aid of a computer (see Figure \ref{fig:theorem_1,(f,z)z,g=1,2,0.01,5_wc}).
    \end{remark}
   
    \begin{figure}[htbp!]
        \centering
        \begin{subfigure}[b]{0.3\textwidth}
            \centering
            \includegraphics[width=\textwidth]{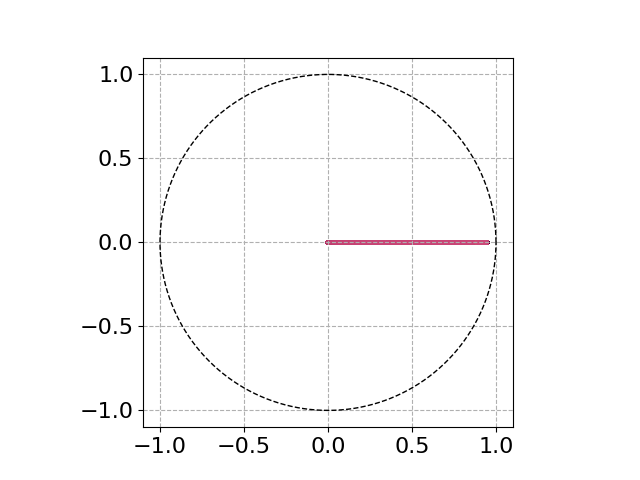}
            \caption{\centering$\gamma=0.01$}
            \label{fig:th_1_g=0.01_fs=16_plasma.png}
        \end{subfigure}
        \hfill
        \begin{subfigure}[b]{0.3\textwidth}
            \centering
            \includegraphics[width=\textwidth]{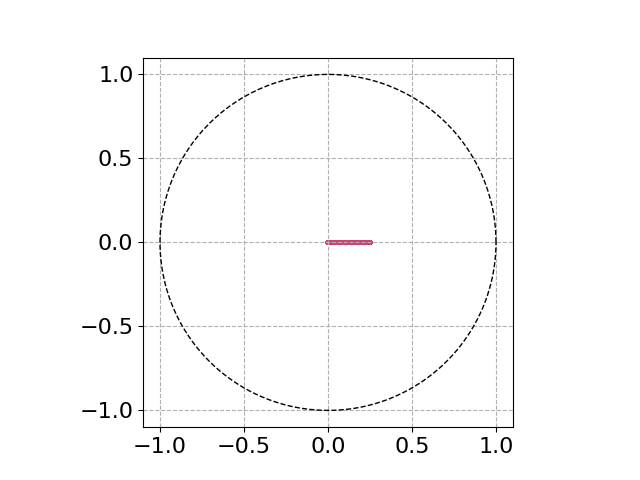}
            \caption{\centering $\gamma=1$ (Hardy space)}
            \label{fig:th_1_g=1_fsh=16_plasma.png}
        \end{subfigure}
        \hfill
        \begin{subfigure}[b]{0.3\textwidth}
            \centering
            \includegraphics[width=\textwidth]{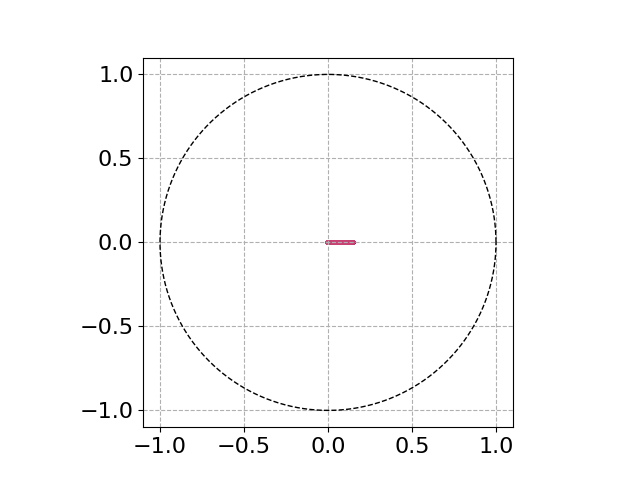}
            \caption{$\gamma=2$ (Bergman space)}
            \label{fig:th_1_g=2_fs=16_plasma.png}
        \end{subfigure}          
    \caption{Comparison of Berezin ranges of the operator $T(f)=\langle f,z \rangle z$ acting on the weighted Hardy space $\mathcal{H}_\gamma (\mathbb{D})$ with different weights.}
    \label{fig:theorem_1,(f,z)z,g=1,2,0.01,5_wc}
    \end{figure}

    The next theorem is a generalization of Theorem \ref{th:(f,z)z}.

    \begin{theorem}\label{th:same integral power}
        Let $T$ be a rank-one operator of the form $T(f)=\langle f,z^n\rangle z^n,$ where $n\in \mathbb{N},$ acting on $\mathcal{H}_\gamma(\mathbb{D}).$ Then, $\textnormal{Ber}(T)=\left[0,\;\dfrac{\gamma^\gamma n^n}{\left(n+\gamma\right)^{n+\gamma}}\right],$ which is convex in $\mathbb{C}.$
    \end{theorem}
    \begin{proof}
        For $\lambda\in \mathbb{D}$, we have
        \begin{equation*}
            T(k_\lambda)=\langle k_\lambda,z^n\rangle z^n=\overline{\langle z^n,k_\lambda\rangle}z^n=\overline{\lambda^n}z^n,
        \end{equation*}
        so that the Berezin transform of $T$ at $\l$ is
        \begin{equation*}
            \widetilde{T}(\lambda)=\langle T\widehat{k}_\lambda,\widehat{k}_\lambda\rangle=\dfrac{1}{\|k_\l\|^2}\langle Tk_\lambda,k_\lambda\rangle=\left(1-|\lambda|^2\right)^\gamma (Tk_\lambda)(\lambda)=\left(1-|\lambda|^2\right)^\gamma |\lambda|^{2n},
        \end{equation*}
        where $|\lambda|\in[0,1).$ We observe that $\widetilde{T}(\lambda)=\widetilde{T}(|\lambda|)$ is a real analytic function. Now we search for the extreme points of $\widetilde{T}(\lambda)$. We differentiate $\widetilde{T}(\lambda)$ with respect to $|\lambda|$ to have
        \begin{equation*}
            \dfrac{d}{d|\lambda|}\left[\widetilde{T}(\lambda)\right]=2|\lambda|^{2n-1}\left(1-|\lambda|^2\right)^{\gamma-1}\left[n-(n+\gamma)|\lambda|^2\right].
        \end{equation*}
        Clearly, $ \dfrac{d}{d|\lambda|}\left[\widetilde{T}(\lambda)\right]=0,\;\text{when}\; |\lambda|^2=0,\;\dfrac{n}{n+\gamma}.$\\
        Now, $\left.\widetilde{T}(\lambda)\right \vert_{|\lambda|^2=0}=0.$ Also, $\left.\widetilde{T}(\lambda)\right \vert_{|\lambda|^2=\frac{n}{n+\gamma}}=\dfrac{\gamma^\gamma n^n}{(n+\gamma)^{n+\gamma}}.$\\
        Thus, $\left.\widetilde{T}(\lambda)\right\vert_{\text{min}}=0$ and $\left.\widetilde{T}(\lambda)\right\vert_{\text{max}}=\dfrac{\gamma^\gamma n^n}{(n+\gamma)^{(n+\gamma)}}.$ Since $T$ is bounded, $\widetilde{T}(\lambda)$ is a continuous real analytic function. Hence, by the intermediate value theorem, $\widetilde{T}(\lambda)$ takes on any real value between $0$ and $\dfrac{\gamma^\gamma n^n}{\left(n+\gamma\right)^{n+\gamma}}$ so that $\textnormal{Ber}(T)=\left[0,\;\dfrac{\gamma^\gamma n^n}{\left(n+\gamma\right)^{n+\gamma}}\right],$ which is a convex set in $\mathbb{C}.$
    \end{proof}
    \begin{remark}
        Putting $\gamma=1$ and $\gamma=2$ respectively in Theorem \ref{th:same integral power}, we see that the Berezin ranges of the rank-one operator $T$ given by $T(f)=\langle f,z^n\rangle z^n$ acting on the Hardy space $H^2(\mathbb{D})$ and the Bergman space $A^2(\mathbb{D})$ are, respectively, $\left[0,\;\dfrac{n^n}{(n+1)^{n+1}}\right]$ and $\left[0,\;\dfrac{4n^n}{(n+2)^{n+2}}\right],$ both of which are convex sets in $\mathbb{C}$, which have been proved recently by Augustine et al. in \cite{Athul2_Article_2024}.
    \end{remark}
    \begin{remark}
        Similarly to what is mentioned in Remark \ref{re:th_1}, for a fixed $n,\;\textnormal{Ber}(T)$ for the operator $T(f)=\langle f,z^n \rangle$ acting on $\mathcal{H}_\gamma(\mathbb{D})$ is contracting with increasing $\gamma,$ and finally when $\gamma \rightarrow \infty$, it tends to coincide with the origin. For a visual verification, see Figure \ref{fig:th_2_n=3_g=0.5,1,2,10.png}.
    \end{remark}

    \begin{figure}[htbp!]
        \centering
        \begin{subfigure}[b]{0.3\textwidth}
            \centering
            \includegraphics[width=\textwidth]{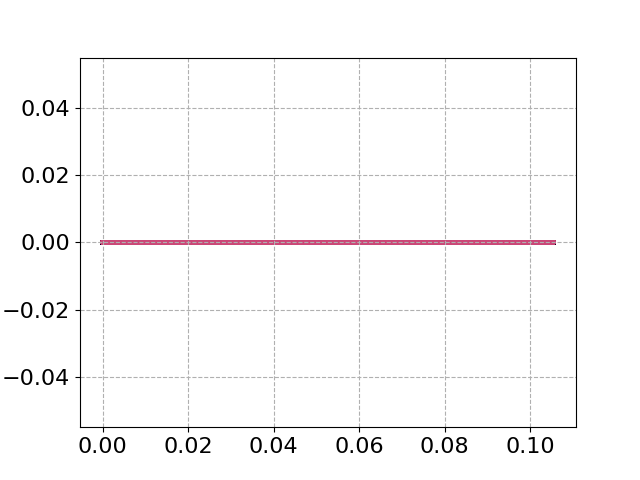}
            \caption{$\gamma=1$ (Hardy space)}
            \label{fig:th_2_g=1_n=3_fs=16_plasma.png}
        \end{subfigure}
        \hfill
        \begin{subfigure}[b]{0.3\textwidth}
            \centering
            \includegraphics[width=\textwidth]{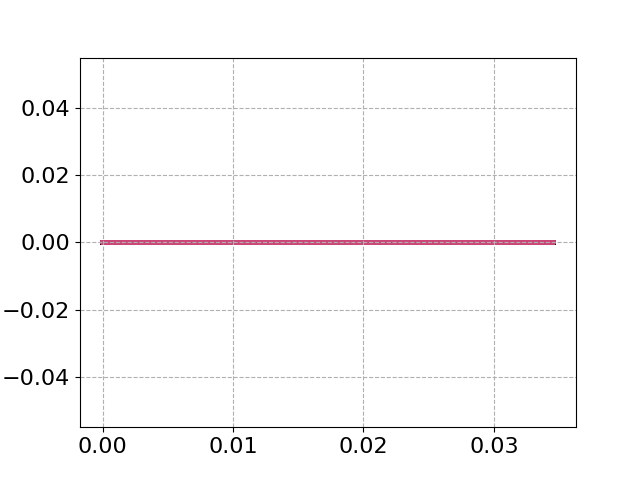}
            \caption{$\gamma=2$ (Bergman space)}
            \label{fig:th_2_g=2_n=3_fs=16_plasma.png}
        \end{subfigure}
        \hfill
        \begin{subfigure}[b]{0.3\textwidth}
            \centering
            \includegraphics[width=\textwidth]{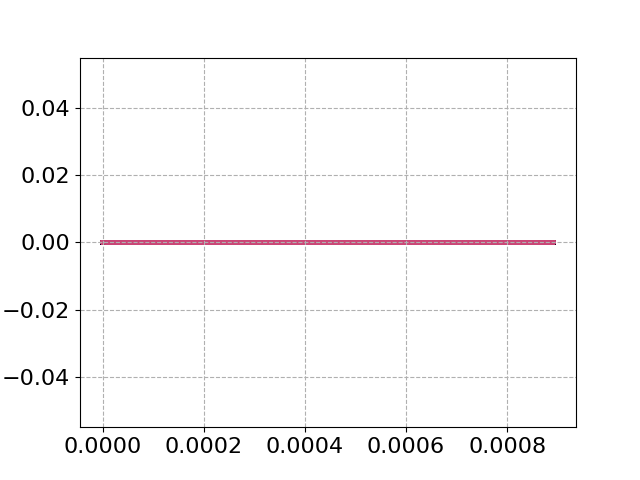}
            \caption{$\gamma=10$}
            \label{fig:th_2_g=10_n=3_fs=16_plasma.png}
        \end{subfigure}
    \caption{Comparison of Berezin ranges of the operator $T(f)=\langle f,z^3 \rangle z^3$ acting on the weighted Hardy space $\mathcal{H}_\gamma(\mathbb{D})$, for different values of $\gamma.$}
    \label{fig:th_2_n=3_g=0.5,1,2,10.png}
    \end{figure}

    In the following, we investigate the convexity of the Berezin range of a finite-rank operator of the form $T_n(f)=\sum_{i=1}^n \langle f,a_iz^i \rangle a_iz^i$, where for all $i,\;a_i\in \mathbb{C},$ acting on $\mathcal{H}_\gamma(\mathbb{D}).$
    
    \begin{theorem}\label{th:finite sum of inner product}
        Let $T_n(f)=\sum_{i=1}^{n}\langle f,a_iz^i\rangle a_iz^i$ be a finite-rank operator acting on the weighted Hardy space $\mathcal{H}_\gamma(\mathbb{D}),$ where $a_i\in \mathbb{C}.$ Then the Berezin range of $T_n$ is convex in $\mathbb{C}.$
    \end{theorem}
    \begin{proof}
        For $\lambda\in \mathbb{D}$, we have
        \begin{equation*}
            T_n(k_\lambda)=\sum_{i=1}^{n}\langle k_\l,a_iz^i\rangle a_iz^i=\sum_{i=1}^{n}|a_i|^2\overline{\lambda}^iz^i,
        \end{equation*}
        so that the Berezin transform of $T$ at $\l$ is
        \begin{equation*}
        \begin{split}
             \widetilde{T_n}(\lambda)&=\langle T_n\widehat{k}_\lambda,\widehat{k}_\lambda\rangle \\
             &=\dfrac{1}{\|k_\l\|^2}\langle T_nk_\lambda,k_\lambda\rangle \\
             &=\left(1-|\lambda|^2\right)^\gamma \left(T_nk_\lambda\right)(\lambda)\\
             &=\left(1-|\lambda|^2\right)^\gamma \sum_{i=1}^{n}|a_i|^2|\lambda|^{2i}.
        \end{split}
      \end{equation*}
      
     Since $T_n$ is a finite rank operator, $\|T_n\|$ is finite. Then, since for any $\lambda\in \mathbb{D},\;\widetilde{T_n}(\lambda)\leq \sup_{\lambda\in \mathbb{D}}|\widetilde{T_n}(\lambda)| = \textnormal{ber}(T_n)\leq \|T_n\|,\;\widetilde{T_n}$ is a bounded function on $\mathbb{D}.$ Also, since $\mathcal{H}_\gamma(\mathbb{D})$ is the collection of holomorphic functions on $\mathbb{D}$ and since $\widetilde{T_n}(\lambda)=\widetilde{T_n}(|\lambda|),\;\widetilde{T_n}$ is a real analytic function and hence is continuous. \par Thus, the Berezin transform of $T_n$ is a real continuous function in the complex plane. Therefore, $\textnormal{Ber}(T_n)$ is convex in $\mathbb{C}.$ \par
    We note that for any $\lambda\in [0,1),\;\widetilde{T_n}(\lambda)\geq0;$ equality occurs when $|\lambda|=0$. Also, when $|\lambda|\rightarrow 1,\;\widetilde{T_n}(\lambda)\rightarrow 0.$ So, the maximum value of $\widetilde{T_n}(\lambda)$ occurs somewhere in $(0,1)$ and it depends on specific values of $a_i$'s. Thus, $\textnormal{Ber}(T_n)=\left[0,\;\text{max}_{\lambda\in (0,1)}\left \{\left(1-|\lambda|^2\right)^\gamma \sum_{i=1}^{n}|a_i|^2|\lambda|^{2i}\right \}\right].$
    \end{proof}
    \begin{remark}
        On putting $\gamma=1$ and $\gamma=2$ respectively in Theorem \ref{th:finite sum of inner product}, we see that the Berezin ranges of the finite rank operator $T_n$ given by $T_n(f)=\sum_{i=1}^n\langle f,a_iz^i\rangle z^i$ acting on the Hardy space $H^2(\mathbb{D})$ and the Bergman space $A^2(\mathbb{D})$ are $\left[0,\;\mbox{max}_{\lambda\in (0,1)}\left \{\left(1-|\lambda|^2\right) \sum_{i=1}^{n}|a_i|^2|\lambda|^{2i}\right \}\right]$ and $\left[0,\;\mbox{max}_{\lambda\in (0,1)}\left \{\left(1-|\lambda|^2\right)^2 \sum_{i=1}^{n}|a_i|^2|\lambda|^{2i}\right \}\right]$ respectively, both of which are convex sets in $\mathbb{C};$ in both cases, the Berezin ranges depend on the specific set of values of $a_i.$ For visual interpretation, see Figure \ref{fig:theorem_3,g=0.5,1,2,4}.
    \end{remark}
    
    \begin{figure}[htbp]
        \centering
        \begin{subfigure}[b]{0.3\textwidth}
            \centering
            \includegraphics[width=\textwidth]{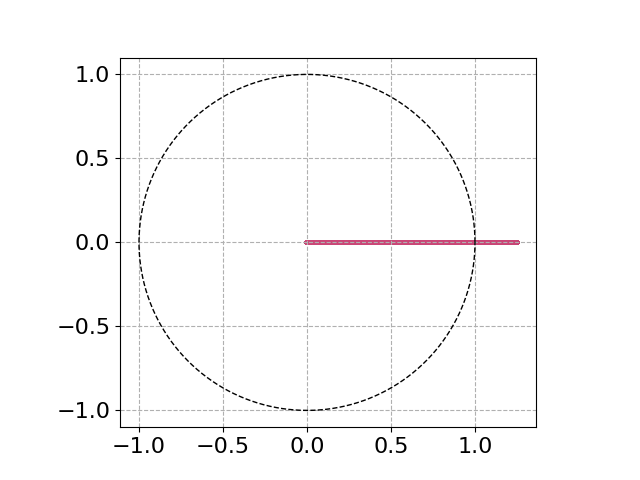}
            \caption{$\gamma=1$ (Hardy space)}
            \label{fig:th_3_g=1__1+i_,_1-i_,_i__fs=16_plasma.png}
        \end{subfigure}
        \hfill
        \begin{subfigure}[b]{0.3\textwidth}
            \centering
            \includegraphics[width=\textwidth]{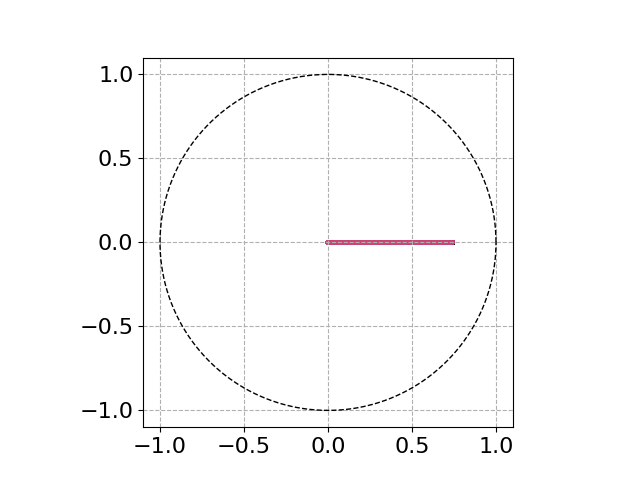}
            \caption{$\gamma=2$ (Bergman space)}
            \label{fig:th_3_g=2__1+i_,_1-i_,_i__fs=16_plasma.png}
        \end{subfigure}
        \hfill
        \begin{subfigure}[b]{0.3\textwidth}
            \centering
            \includegraphics[width=\textwidth]{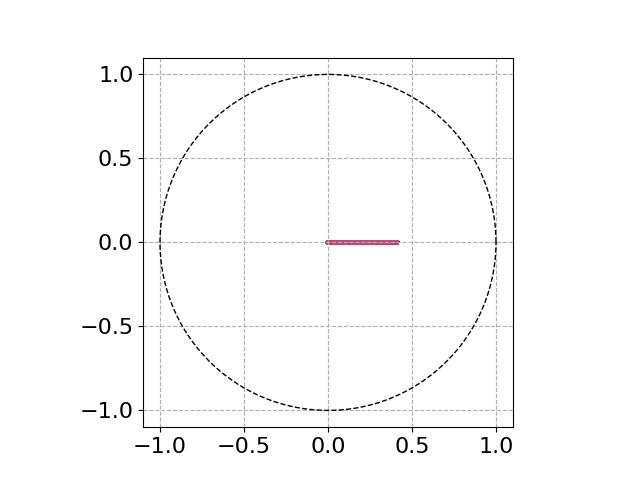}
            \caption{$\gamma=4$}
            \label{fig:th_3_g=4__1+i_,_1-i_,_i__fs=16_plasma.png}
        \end{subfigure}
    \caption{Berezin range of the operator $T(f)=\displaystyle\sum_{i=1}^{3}\langle f,a_iz^i\rangle a_iz^i$ acting on the weighted Hardy space $\mathcal{H}_\gamma(\mathbb{D})$ for different values of $\gamma,$ where $a_1=1+i,\;a_2=1-i,\;a_3=i.$}
    \label{fig:theorem_3,g=0.5,1,2,4}
    \end{figure}
    
    Now consider a compact self-adjoint operator of the form $T(f)=\sum_{n=1}^{\infty}{\langle f,az^n\rangle az^n}$, where $a\in \mathbb{D},$ acting on $\mathcal{H}_\gamma(\mathbb{D}).$ In the following theorem, we calculate the Berezin range of this operator.
    \begin{theorem}\label{th:sum_infty}
        Let $T(f)=\sum_{n=1}^{\infty}{\langle f,az^n\rangle az^n}$, where $a\in \mathbb{D},$ be a compact self-adjoint operator acting on $\mathcal{H}_\gamma(\mathbb{D}).$ Then the Berzin range of $T$ is given by 
        \begin{equation*}
            \textnormal{Ber}(T)=\begin{cases}
			\left[0,\;|a|^2\dfrac{(\gamma-1)^{\gamma-1}}{\gamma^{\gamma}}\right], & \text{if $\gamma>1$}\\
            \left[0,\;|a|^2\right), & \text{if $\gamma=1$}\\
            [0,\;\infty), & \text{if $0<\gamma<1$}
		 \end{cases};
        \end{equation*}
     in either case, which is convex in $\mathbb{C}.$   
    \end{theorem}
\begin{proof}
    For $\lambda\in \mathbb{D}$, we have
        \begin{equation*}
            T(k_\lambda)=\sum_{n=1}^{\infty}\langle k_\lambda,az^n\rangle az^n=|a|^2\sum_{n=1}^{\infty}\overline{\lambda}^nz^n,
        \end{equation*}
        so that the Berezin transform of $T$ at $\lambda$ is
        \begin{equation*}
        \begin{split}
             \widetilde{T}(\lambda)&=\langle T\widehat{k}_\lambda,\widehat{k}_\lambda\rangle \\
             &=\dfrac{1}{\|k_\l\|^2}\langle Tk_\lambda,k_\lambda\rangle \\
             &=\left(1-|\lambda|^2\right)^\gamma \left(Tk_\lambda\right)(\lambda)\\
             &=\left(1-|\lambda|^2\right)^\gamma |a|^2\sum_{n=1}^{\infty}|\lambda|^{2n}\\
             &=\left(1-|\lambda|^2\right)^\gamma |a|^2 \dfrac{|\lambda|^2}{1-|\lambda|^2}\\
             &=|a|^2|\lambda|^2\left(1-|\lambda|^2\right)^{\gamma-1},
        \end{split}
      \end{equation*}
             where $|\lambda|\in[0,1).$ We observe that $\widetilde{T}(\lambda)=\widetilde{T}(|\lambda|)$ is a real analytic function. Now we search for the extreme points of $\widetilde{T}(\lambda)$. We differentiate $\widetilde{T}(\lambda)$ with respect to $|\lambda|$ to have
        \begin{equation*}
            \dfrac{d}{d|\lambda|}\left[\widetilde{T}(\lambda)\right]=2|a|^2|\lambda|\left(1-|\lambda|^2\right)^{\gamma-2}\left(1-\gamma|\lambda|^2\right).
        \end{equation*}
        \begin{case}
            $\gamma>1.$
        \end{case}
        In this case, $ \dfrac{d}{d|\lambda|}\left[\widetilde{T}(\lambda)\right]=0,\;\text{when}\; |\lambda|^2=0,\;\dfrac{1}{\gamma}.$ \par
        Now, $\left.\widetilde{T}(\lambda)\right \vert_{|\lambda|^2=0}=0.$ Also, $\left.\widetilde{T}(\lambda)\right \vert_{|\lambda|^2=\frac{1}{\gamma}}=|a|^2\dfrac{(\gamma-1)^{\gamma-1}}{\gamma^{\gamma}}.$\par
        Thus, $\left.\widetilde{T}(\lambda)\right\vert_{\text{min}}=0$ and $\left.\widetilde{T}(\lambda)\right\vert_{\text{max}}=|a|^2\dfrac{(\gamma-1)^{\gamma-1}}{\gamma^{\gamma}}.$ Since $T$ is compact, it is bounded; so, $\widetilde{T}(\lambda)$ is a continuous real analytic function. Hence by Intermediate value theorem, $\widetilde{T}(\lambda)$ takes on any real value between $0$ and $|a|^2\dfrac{(\gamma-1)^{\gamma-1}}{\gamma^{\gamma}}$, so that $\textnormal{Ber}(T)=\left[0,\;|a|^2\dfrac{(\gamma-1)^{\gamma-1}}{\gamma^{\gamma}}\right],$ which is a convex set in $\mathbb{C}.$\par
        \begin{case}
            $\gamma=1.$
        \end{case}
        In this case, $\dfrac{d}{d\lambda}\left[\widetilde{T}(\lambda)\right]=2|a|^2|\lambda|>0$ for $|\lambda|>0.$ This shows that the real valued function $\widetilde{T}$ increases with increasing $|\lambda|,$ and finally when $|\lambda|\rightarrow 1,\;\widetilde{T}(\lambda)\rightarrow |a|^2.$ Since $\widetilde{T}$ is continuous, the Berezin range of $T,$ in this case, is given by $\textnormal{Ber}(T)=\left[0,\;|a|^2\right),$ which, being a half open line segment on the real axis is convex in $\mathbb{C}.$ (This result has already been proved in \cite{Athul2_Article_2024})
        \begin{case}
            $0<\lambda<1.$
        \end{case}
        In this case, we observe that for any $\lambda,\;\widetilde{T}(\lambda)\geq0;$ equality occurs when $|\lambda|=0$. Also, when $|\lambda|\rightarrow 1,\;\widetilde{T}(\lambda)\rightarrow \infty.$ Since $\widetilde{T}$ is a continuous real analytic function, $\textnormal{Ber}(T)=[0,\;\infty),$ i.e., the Berezin range of $T$, in this case, is the entire positive real axis, clearly which is convex in $\mathbb{C}.$
    \end{proof}
    \begin{remark}
        Putting $\gamma=2$ in Theorem \ref{th:sum_infty}, we find that the Berzin range of the compact self-adjoint operator $T$ given by $T(f)=\langle f,az^n\rangle az^n,$ where $a\in \mathbb{D},$ acting on the Bergman space is $\left[0,\dfrac{|a|^2}{4}\right]$, which has already been proved in \cite{Athul2_Article_2024}.
        \end{remark}
    For visual verification of Theorem \ref{th:sum_infty}, see Figure \ref{fig:theorem_4_g=1,2,4_a=(1+i)/2}
        
       \begin{figure}[htbp]
           \centering
           \begin{subfigure}[b]{0.3\textwidth}
               \centering
               \includegraphics[width=\textwidth]{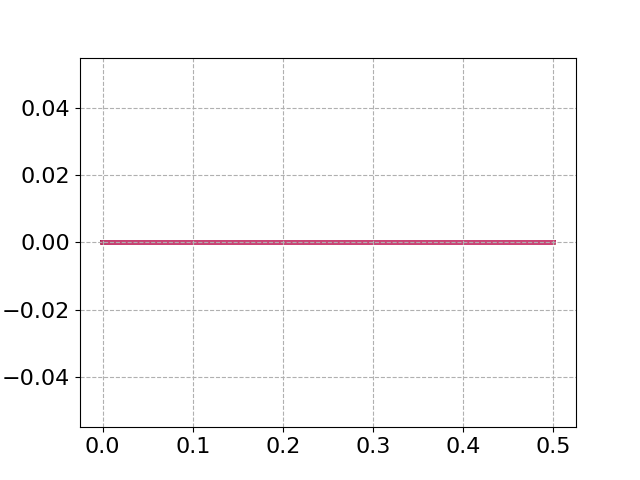}
               \caption{$\gamma=1$ (Hardy space)}
               \label{fig:th_4_g=1_a=_1+i_by_2__fs=16_plasma.png}
           \end{subfigure}
           \hfill
           \begin{subfigure}[b]{0.3\textwidth}
               \centering
               \includegraphics[width=\textwidth]{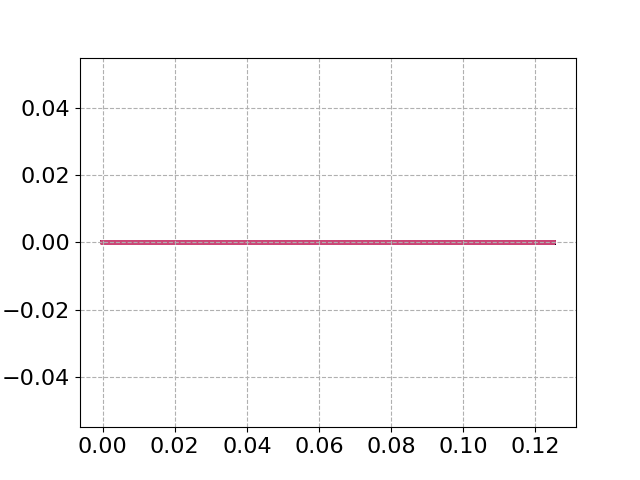}
               \caption{$\gamma=2$ (Bergman space)}
               \label{fig:th_4_g=2_a=_1+i_by_2__fs=16_plasma.png}
           \end{subfigure}
           \hfill
           \begin{subfigure}[b]{0.3\textwidth}
               \centering
               \includegraphics[width=\textwidth]{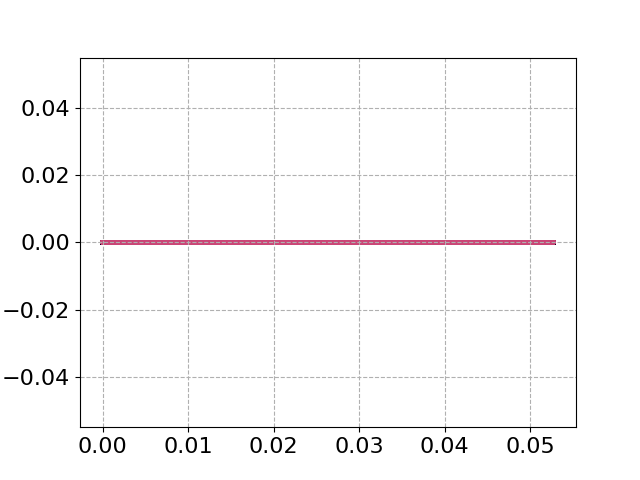}
               \caption{$\gamma=4$}
               \label{fig:th_4_g=4_a=_1+i_by_2__fs=16_plasma.png}
           \end{subfigure}
        \caption{Berezin range of the operator $\displaystyle\sum_{n=1}^{\infty}\langle f,az^n\rangle az^n$ acting on the weighted Hardy space $\mathcal{H}_\gamma(\mathbb{D}),$ for different values of $\gamma,$ where $a=\dfrac{1}{2}(1+i)$. [Note that here $|a|^2=0.5$]}
        \label{fig:theorem_4_g=1,2,4_a=(1+i)/2}
       \end{figure}
             
        In the following, we generalize Theorem \ref{th:finite sum of inner product}, by replacing the set of functions $\left\{a_iz^i\right\}_{i=1}^n$ by an arbitrary set of functions $\left\{g_i(z)\right\}_{i=1}^n$in $\mathcal{H}_\gamma(\mathbb{D}).$
        \begin{theorem}\label{th:generalized finite sum}
            Let for an arbitrary set of functions $g_i,\;i=1(n),$ the operator $T_n$ defined by $T_n(f)=\sum_{i=1}^n  {\langle f, g_i \rangle g_i}$ be a finite-rank operator acting on the weighted Hardy space $\mathcal{H}_\gamma(\mathbb{D})$. Then the Berezin range of $T_n$ is convex in $\mathbb{C}$.
        \end{theorem}
        \begin{proof}
            For $\lambda\in \mathbb{D}$, we have
        \begin{equation*}
            T_n(k_\lambda)=\sum_{i=1}^{n}\langle k_\lambda,g_i\rangle g_i=\sum_{i=1}^{n}\overline{g_i(\lambda)} g_i(z)
        \end{equation*}
        so that the Berezin transform of $T_n$ at $\lambda$ is
        \begin{equation*}
        \begin{split}
             \widetilde{T_n}(\lambda)&=\langle T_n\widehat{k}_\lambda,\widehat{k}_\lambda\rangle \\
             &=\dfrac{1}{\|k_\lambda\|^2}\langle T_nk_\lambda,k_\lambda\rangle \\
             &=\left(1-|\lambda|^2\right)^\gamma \left(T_nk_\lambda\right)(\lambda)\\
             &=\left(1-|\lambda|^2\right)^\gamma \sum_{i=1}^{n}|g_i(\lambda)|^2,
        \end{split}
      \end{equation*}
      where $\lambda \in \mathbb{D},$ i.e., $|\lambda| \in [0,1).$\par
      Since for each $i=1(n),\;g_i \in \mathcal{H}_\gamma(\mathbb{D})$ is a holomorphic function in $\mathbb{D},\;\widetilde{T_n}(\lambda)$ is a continuous function of $\lambda.$ This implies that $\widetilde{T_n}(\lambda)$ is a real continuous function on the open unit disk $\mathbb{D}.$ Therefore, $\mathbb{D}$ being a connected set in $\mathbb{C},$ its image under $\widetilde{T_n}$ must be a connected set in $\mathbb{R}.$ Now, we know that on the real line the only connected subsets are the intervals including degenerate intervals consisting of a single point. Thus, the range of $\widetilde{T_n}$ is an interval on the real line. As an interval is a convex set in the complex plane $\mathbb{C},$ it follows that $\textnormal{Ber}(T_n)$ is convex in $\mathbb{C}.$
        \end{proof}
        In the following, with the aid of computer, we plot the Berezin ranges of the operator $T_2(f)=\langle f,\sin z \rangle \sin z+\langle f,z^2 \rangle z^2$ acting on different spaces (see Figure \ref{fig:(f,sinz)sinz+(f,z^2)z^2}).
        
    \begin{figure}[htbp]
        \centering
        \begin{subfigure}[b]{0.3\textwidth}
            \centering
            \includegraphics[width=\textwidth]{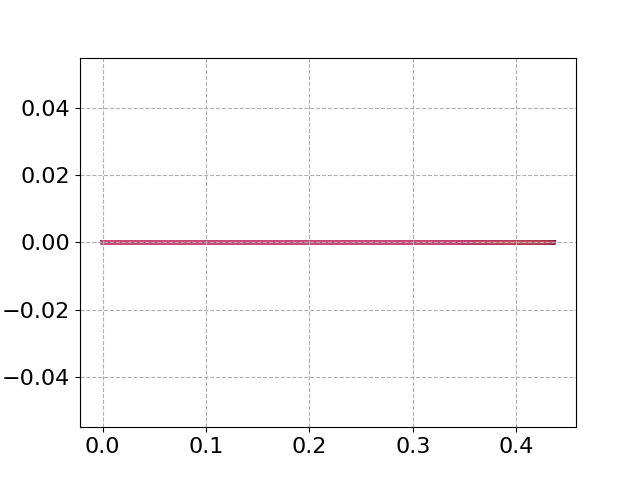}
            \caption{$\gamma=1$ (Hardy space)}
            \label{fig:th_5_g=1_sinz,z_2_fs=16_plasma.png}
        \end{subfigure}
        \hfill
        \begin{subfigure}[b]{0.3\textwidth}
            \centering
            \includegraphics[width=\textwidth]{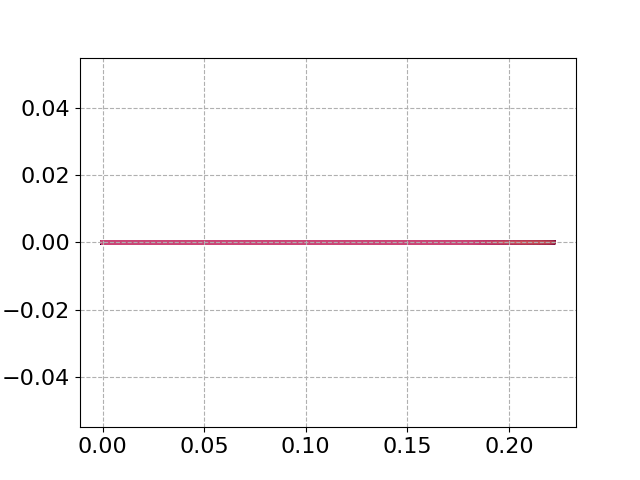}
            \caption{$\gamma=2$ (Bergman space)}
            \label{fig:th_5_g=2_sinz,z_2_fs=16_plasma.png}
        \end{subfigure}
        \hfill
        \begin{subfigure}[b]{0.3\textwidth}
            \centering
            \includegraphics[width=\textwidth]{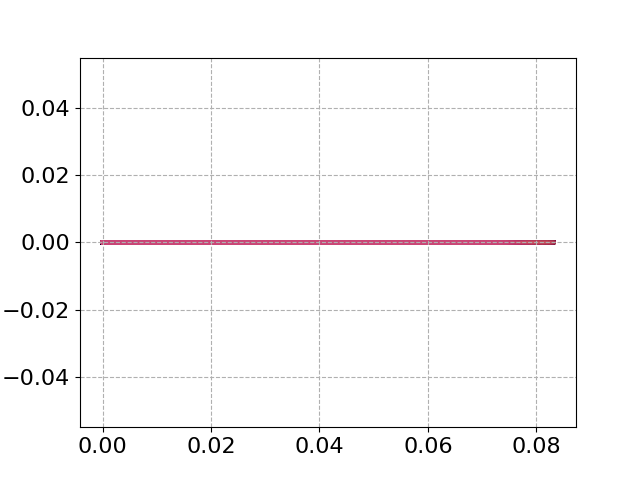}
            \caption{$\gamma=5$}
            \label{fig:th_5_g=5_sinz,z_2_fs=16_plasma.png}
        \end{subfigure}
    \caption{Berezin ranges of $T_2(f)=\langle f,\sin z \rangle \sin z+\langle f,z^2 \rangle z^2$ acting on $\mathcal{H}_\gamma(\mathbb{D}),$ for different values of $\gamma.$}
    \label{fig:(f,sinz)sinz+(f,z^2)z^2}
    \end{figure}
    
  \par Now consider a more generalization of the operator $T_n$ described in Theorem \ref{th:generalized finite sum} by taking $T_n(f)=\sum_{i=1}^{n}\langle f,g_i\rangle h_i$ for arbitrary set of functions $\left\{g_i\right\},\left\{h_i\right\}$ in $\mathcal{H}_\gamma(\mathbb{D}).$ In the following, we will show that the Berezin range of an operator of this form is not always convex. To do this, first, we prove the following theorem and its corollary.
   \begin{theorem}\label{th:more generalized finite sum}
       Let $T_n(f)=\sum_{i=1}^n  {\langle f, g_i \rangle h_i},$ where $f,g_i,h_i\in \mathcal{H}_\gamma(\mathbb{D})$ be a finite-rank operator acting on the weighted Hardy space $\mathcal{H}_\gamma(\mathbb{D})$. If all the functions $g_i$ and $h_i$ have power series representations about the origin with real coefficients, then the Berezin range of $T_n,\; \textnormal{Ber}(T_n)$ is closed under complex conjugation and hence is symmetric with respect to the real axis.
   \end{theorem}
   \begin{proof}
       Let $g_i=\sum_{m=0}^{\infty}a_{i,m}z^m$ and $h_i=\sum_{m=0}^{\infty}b_{i,m}z^m$, where $a_{i,m},\;b_{i,m}\in\mathbb{R};\;i=1(n).$ Then, for any $\lambda\in\mathbb{D},$ we have
       \begin{equation*}
            T_n(k_\lambda)=\sum_{i=1}^{n}\langle k_\lambda,g_i\rangle h_i=\sum_{i=1}^{n}\overline{g_i(\lambda)} h_i(z)=\sum_{i=1}^{n}\left(\overline{\sum_{m=0}^{\infty}a_{i,m}\lambda^m}\sum_{m=0}^{\infty}b_{i,m}z^m\right),
        \end{equation*}
        so that the Berezin transform of $T_n$ at $\lambda$ is
        \begin{equation*}
        \begin{split}
             \widetilde{T_n}(\lambda)&=\langle T_n\widehat{k}_\lambda,\widehat{k}_\lambda\rangle \\
             &=\dfrac{1}{\|k_\lambda\|^2}\langle T_nk_\lambda,k_\lambda\rangle \\
             &=\left(1-|\lambda|^2\right)^\gamma \left(T_nk_\lambda\right)(\lambda)\\
             &=\left(1-|\lambda|^2\right)^\gamma \sum_{i=1}^{n}\left(\overline{\sum_{m=0}^{\infty}a_{i,m}\lambda^m}\sum_{m=0}^{\infty}b_{i,m}\lambda^m\right),
        \end{split}
      \end{equation*}
      where $\lambda\in\mathbb{D}.$ Then,
      \begin{equation*}
      \begin{split}
           \overline{\widetilde{T_n}(\lambda)}&=\left(1-|\overline{\lambda}|^2\right)^\gamma \overline{\sum_{i=1}^{n}\left(\overline{\sum_{m=0}^{\infty}a_{i,m}\lambda^m}\sum_{m=0}^{\infty}b_{i,m}\lambda^m\right)}\\
           &=\left(1-|\lambda|^2\right)^\gamma\sum_{i=1}^{n}\left(\overline{\overline{\sum_{m=0}^{\infty}a_{i,m}\lambda^m}\sum_{m=0}^{\infty}b_{i,m}\lambda^m}\right)\\
           &=\left(1-|\lambda|^2\right)^\gamma\sum_{i=1}^n\left(\sum_{m=0}^\infty a_{i,m}\lambda^m\sum_{m=0}^\infty b_{i,m}\overline{\lambda}^m\right)\,;~~\mbox{since}\;a_{i,m},\;b_{i,m}\in \mathbb{R};\;i=1(n).\\
           &=\left(1-|\lambda|^2\right)^\gamma\sum_{i=1}^n\left(\overline{\sum_{m=0}^\infty a_{i,m}\overline{\lambda}^m}\sum_{m=0}^\infty b_{i,m}\overline{\lambda}^m\right)\\
           &=\widetilde{T_n}\left(\overline{\lambda}\right).
      \end{split}
    \end{equation*}
    \par Since $\lambda\in \mathbb{D},\;\overline{\lambda}\in\mathbb{D}.$ So, $\widetilde{T_n}\left(\overline{\lambda}\right)\in \textnormal{Ber}(T_n),$ i.e., $\overline{\widetilde{T_n}(\lambda)}\in Ber(T_n).$ This shows that the conjugate of any number in $\textnormal{Ber}(T_n)$ is also in $\textnormal{Ber}(T_n).$ In other words, $\textnormal{Ber}(T_n)$ is closed under complex conjugation.    
   \end{proof}
   \begin{corollary}\label{cor:real_part}
       Let $T_n(f)=\sum_{i=1}^n  {\langle f, g_i \rangle h_i},$ where $f,g_i,h_i\in \mathcal{H}_\gamma(\mathbb{D})$ be a finite-rank operator acting on the weighted Hardy space $\mathcal{H}_\gamma(\mathbb{D})$. If all the functions $g_i$ and $h_i$ have power series representations about the origin with real coefficients, then the following statement holds. If the Berezin range of $T_n,\;\textnormal{Ber}(T_n)$ is convex in $\mathbb{C},$, then $\mathcal{R}\left\{\widetilde{T_n}(\lambda)\right\}\in \textnormal{Ber}(T_n)$ for all $\lambda\in \mathbb{D},$ where $\mathcal{R}\left\{\widetilde{T_n}(\lambda)\right\}$ denotes the real part of $\widetilde{T_n}(\lambda).$
    \end{corollary}
    \begin{proof}
        Suppose, $\textnormal{Ber}(T_n)$ be convex in $\mathbb{C}$. Let $\lambda\in \mathbb{D}.$ Since by Theorem \ref{th:more generalized finite sum}, $\textnormal{Ber}(T_n)$ is closed under complex conjugation, we find that $\overline{\widetilde{T_n}(\lambda})\in \textnormal{Ber}(T_n).$ Then, as $\textnormal{Ber}(T_n)$ is convex in $\mathbb{C},$ we have
        \begin{equation*}
            \dfrac{1}{2}\widetilde{T_n}(\lambda)+\dfrac{1}{2}\overline{\widetilde{T_n}(\lambda)}=\mathcal{R}\left\{\widetilde{T_n}(\lambda)\right\}\in \textnormal{Ber}(T_n).
        \end{equation*}
    \end{proof}
    With the aid of Corollary \ref{cor:real_part}, we are now in a position to show that the Berezin range of a finite-rank operator of the form $T(f)=\sum_{i=1}^n\langle f,g_i\rangle h_i;\;g_i,h_i\in \mathcal{H}_\gamma(\mathbb{D}),$ acting on the weighted Hardy space $\mathcal{H}_\gamma(\mathbb{D})$ is not always convex.
    \begin{example}
        Let $T(f)=\langle f,1-z\rangle (1-z^2),\;f\in \mathcal{H}_\gamma(\mathbb{D}),$ be a finite-rank operator acting on $\mathcal{H}_\gamma(\mathbb{D}).$ For $\lambda\in \mathbb{D},$ we have
        \begin{equation*}
            T(k_\lambda)=\langle k_\lambda,1-z\rangle (1-z^2)=\overline{\langle 1-z,k_\lambda\rangle}(1-z^2)=\left(\overline{1-\lambda}\right)\left(1-z^2\right)=\left(1-\overline{\lambda}\right)\left(1-z^2\right).
        \end{equation*}
        \par Therefore, the Berezin transform of $T$ at $\lambda$ is given by
        \begin{equation*}
            \begin{split}
                \widetilde{T}(\lambda)&=\langle T\widehat{k}_\lambda,\widehat{k}_\lambda\rangle\\
                &=\dfrac{1}{\|k_\lambda\|^2}\langle Tk_\lambda,k_\lambda \rangle \\
                &=\left(1-|\lambda|^2\right)^\gamma \left(Tk_\lambda\right)(\lambda)\\
                &=\left(1-|\lambda|^2\right)^\gamma \left(1-\overline{\lambda}\right)(1-\lambda^2)\\
                &=\left(1-|\lambda|^2\right)^\gamma \left(1-\overline{\lambda}-\lambda^2+|\lambda|^2 \lambda\right).
            \end{split}
        \end{equation*}
        \par On putting $\lambda=x+iy\;\left(x^2+y^2<1\right),$ we have
        \begin{equation*}
            \widetilde{T}(\lambda)=\left(1-x^2-y^2\right)^\gamma \left[\left\{1-x-x^2+x^3+y^2(1+x)\right\}+iy\left\{1-2x+x^2+y^2\right\}\right],
        \end{equation*}
        so that
        \begin{equation*}
            \mathcal{R}\left\{\widetilde{T}(\lambda)\right\}=\left(1-x^2-y^2\right)^\gamma \left\{1-x-x^2+x^3+y^2(1+x)\right\},
        \end{equation*}
        and
        \begin{equation*}
            \mathcal{I}\left\{\widetilde{T}(\lambda)\right\}=\left(1-x^2-y^2\right)^{\gamma}y\left\{1-2x+x^2+y^2\right\}.
        \end{equation*}
        \par We claim that $\textnormal{Ber}(T)$ is not convex in $\mathbb{C}.$ Suppose, on the contrary, $\textnormal{Ber}(T)$ be convex. Then by Corollary \ref{cor:real_part}, $\mathcal{R}\left\{\widetilde{T}(z)\right\}\in \textnormal{Ber}(T)$ for all $z\in \mathbb{D}.$ This infers that for an arbitrary $z\in \mathbb{D},\;\text{there exists} \;\lambda\in \mathbb{D}$ such that $\widetilde{T}(\lambda)=\mathcal{R}\left\{\widetilde{T}(z)\right\};$ which, in turn, gives
        \begin{equation*}
            \mathcal{I}\left\{\widetilde{T}(\lambda)\right\}=\left(1-x^2-y^2\right)^{\gamma}y\left\{1-2x+x^2+y^2\right\}=0.
        \end{equation*}
        Notice that $1-2x+x^2+y^2=0$ implies, $(x-1)^2+y^2=0,$ which further implies that $x=1,\;y=0,$ i.e., the point $(x,y)$ lies on the unit circle $\mathbb{T},$ which is not true. So, $\mathcal{I}\left\{\widetilde{T}(\lambda)\right\}=0$ if and only if $y=0.$ With $y=0,$ we obtain $\mathcal{R}\left\{\widetilde{T}(z)\right\}=\mathcal{R}\left\{\widetilde{T}(\lambda)\right\}=(1-x)\left(1-x^2\right)^{\gamma+1}.$ 
        \par We set $\gamma=0.1$. Using calculus, after some calculations we can see that the maximum value of $\mathcal{R}\left\{\widetilde{T}(z)\right\}$ occurs at $x=-\dfrac{5}{16},$ and its maximum value is $1.17222$ (correct upto six significant figures). Since $z$ is arbitrary, we have for every $z\in \mathbb{D},$
        \begin{equation*}
            \mathcal{R}\left\{\widetilde{T}(z)\right\}<1.18<1.2.
        \end{equation*}
        But for $z=-0.1+0.5i,\;\mathcal{R}\left\{\widetilde{T}(z)\right\}=1.27502$ (correct upto six significant figures). This yields a contradiction. Hence, $\textnormal{Ber}(T)$ is not convex (see Figure \ref{fig:Example3.1,g=0.1}).
    \end{example}
    \begin{remark}
        With the aid of computer, we can see that the non-convexity nature of the Berezin range of the operator $T(f)=\langle f,1-z \rangle \left(1-z^2 \right)$ acting on $\mathcal{H}_\gamma(\mathbb{D})$ gradually disappears with increasing $\gamma,$ which also indicates that convexity of Berezin range of an operator does not solely depend on the operator but also depends on the space on which the operator acts (see Figure \ref{fig:comparison}).
    \end{remark}

    \begin{figure}[htbp!]
        \centering
         \includegraphics[width=0.5\textwidth]{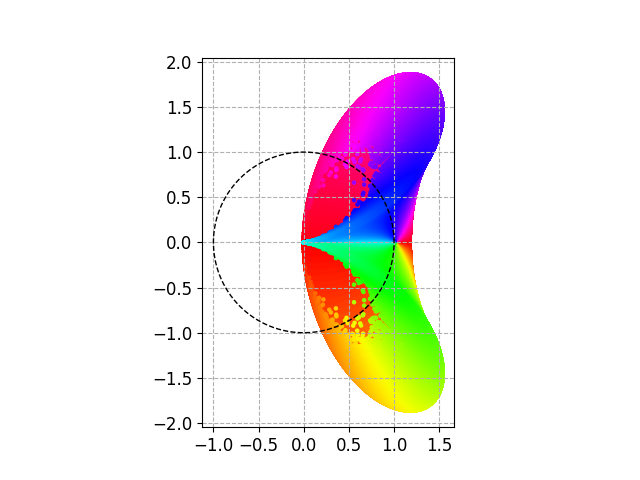}
        \caption{Berezin range of the operator $T(f)=\langle f,1-z \rangle (1-z^2)$ acting on $\mathcal{H}_\gamma(\mathbb{D})$ with $\gamma=0.1$ (apparently not convex).}
        \label{fig:Example3.1,g=0.1}
    \end{figure}

    \begin{figure}[htbp!]
    \centering
        \begin{subfigure}[b]{0.3\textwidth}
         \centering
         \includegraphics[width=\textwidth]{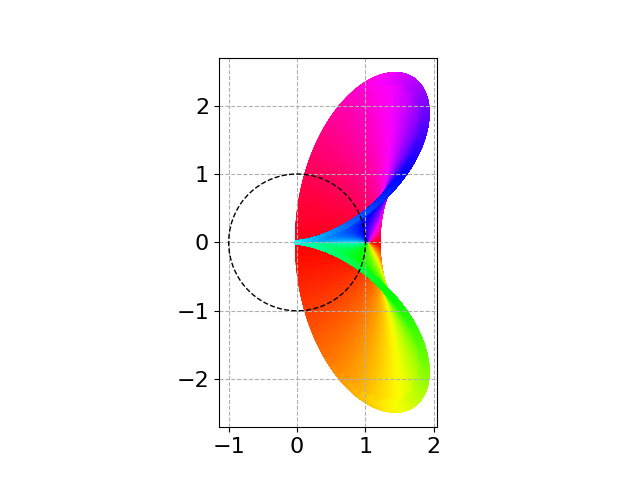}
         \caption{\centering $\gamma=0.01$ (Apparently not convex)}
         \label{fig:ex_3.1_g=0.01_fs=16.png}
        \end{subfigure}
     \hfill
        \begin{subfigure}[b]{0.3\textwidth}
         \centering
         \includegraphics[width=\textwidth]{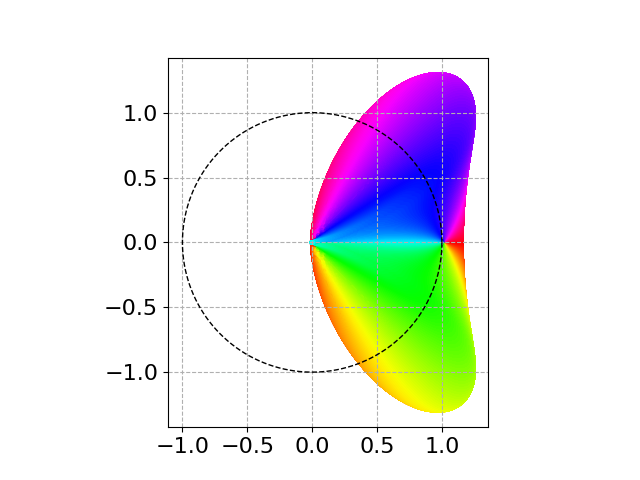}
         \caption{\centering $\gamma=0.3$ (Apparently not convex)}
         \label{fig:ex_3.1_g=0.3_fs=16.png}
        \end{subfigure}
     \hfill
        \begin{subfigure}[b]{0.3\textwidth}
         \centering
         \includegraphics[width=\textwidth]{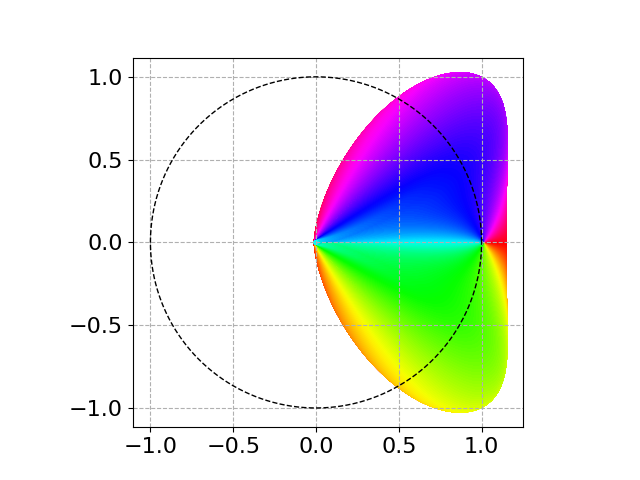}
         \caption{\centering $\gamma=0.5$ (Apparently not convex)}
         \label{fig:ex_3.1_g=0.5_fs=16.png}
        \end{subfigure}
     \hfill
        \begin{subfigure}[b]{0.3\textwidth}
         \centering
         \includegraphics[width=\textwidth]{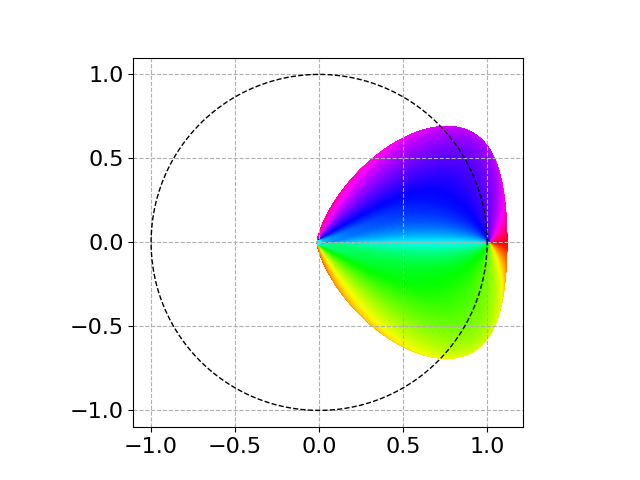}
         \caption{\centering $\gamma=1$ (Apparently convex)}
         \label{fig:ex_3.1_g=1_fs=16.png}
        \end{subfigure}
     \hfill
        \begin{subfigure}[b]{0.3\textwidth}
         \centering
         \includegraphics[width=\textwidth]{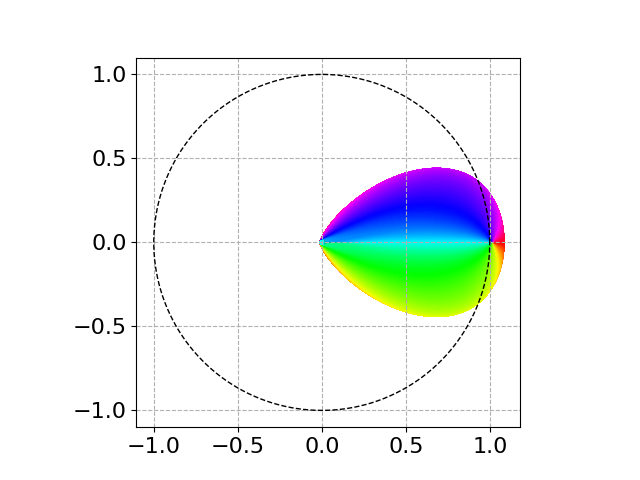}
         \caption{\centering $\gamma=2$ (Apparently convex)}
         \label{fig:ex_3.1_g=2_fs=16.png}
         \end{subfigure}
     \hfill
        \begin{subfigure}[b]{0.3\textwidth}
         \centering
         \includegraphics[width=\textwidth]{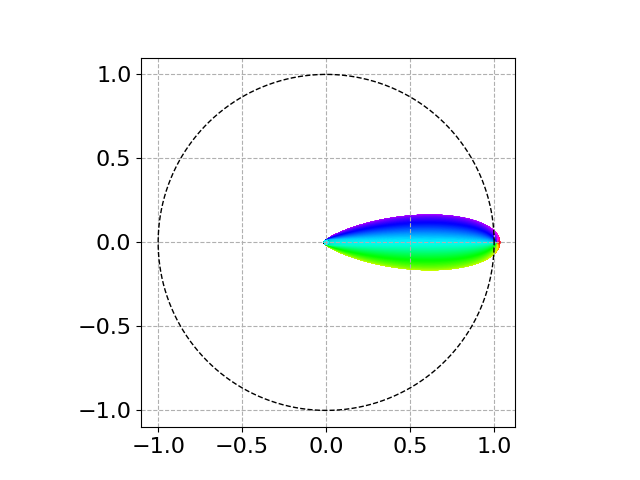}
         \caption{\centering $\gamma=10$ (Apparently convex)}
         \label{fig:ex_3.1_g=10_fs=16.png}
        \end{subfigure}
    \caption{Comparison of Berezin ranges of the operator $T(f)=\langle f,1-z \rangle \left(1-z^2 \right)$ on different spaces.}
    \label{fig:comparison}
    \end{figure}    
    
\begin{example}
		Let $P$ be the projection of  $\mathcal{H}_\gamma(\mathbb{D})$ onto the subspace spanned by $z^k, k\geq 0$. That is, $Pf=(k+1)\langle f, z^k\rangle z^k$. Then $\|P\|=1$ and for every $\lambda\in \mathbb{D},$ we have
        \begin{equation*}
            P(k_\lambda)=(k+1)\langle k_\lambda,z^k\rangle z^k=(k+1)\overline{\langle z^k,k_\lambda\rangle}z^k=(k+1)\overline{\lambda}^kz^k
        \end{equation*}
        so that the Berezin transform of $P$ is given by
        \begin{equation*}
        \begin{split}
            \widetilde{P}(\lambda)&=\langle P\widehat{k}_\lambda,\widehat{k}_\lambda\rangle\\
            &=\dfrac{1}{\|k_\lambda\|^2}\langle Pk_\lambda,k_\lambda\rangle\\
            &=\left(1-|\lambda|^2\right)^\gamma\left(Pk_\lambda\right)(\lambda)\\
            &=(k+1)(1-|\l|^2)^\gamma|\l|^{2k}.
        \end{split}
        \end{equation*}
        Clearly, the Berezin transform of $P$ is a real continuous analytic function and hence the Berezin range of $P,\;\textnormal{Ber}(P)=\left\{\widetilde{P}(\lambda):\lambda\in \mathbb{D}\right \}$ is convex in $\mathbb{C}.$ For a visual verification (see Figure \ref{fig:Example_3.2}).
        \end{example} 

\begin{figure}[htbp!]
     \centering
     \begin{subfigure}[b]{0.3\textwidth}
         \centering
         \includegraphics[width=\textwidth]{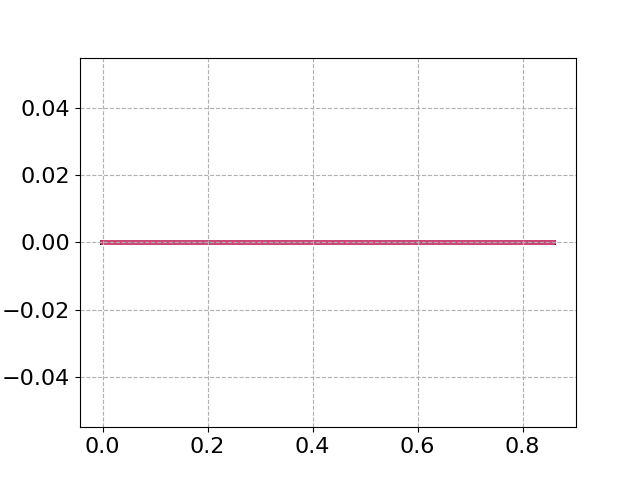}
         \caption{$\gamma=0.5$}
         \label{fig:ex_3.2_g=0.5_fs=16_plasma_nc_nasp.png}
     \end{subfigure}
    \hfill
     \begin{subfigure}[b]{0.3\textwidth}
         \centering
         \includegraphics[width=\textwidth]{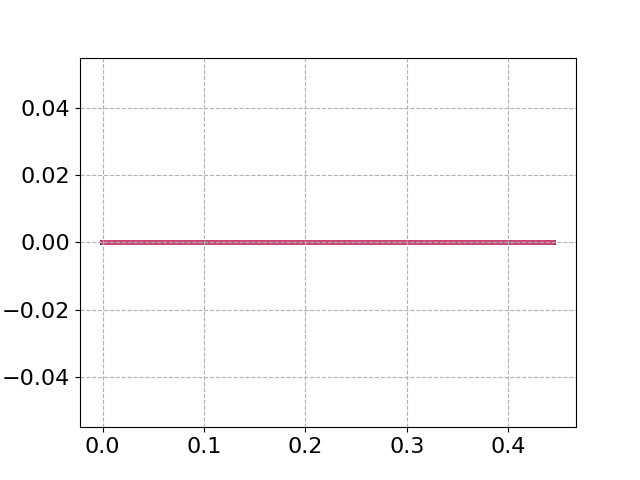}
         \caption{$\gamma=1$ (Hardy space)}
         \label{fig:ex_3.2_g=1_fs=16_plasma_nc_nasp.png}
     \end{subfigure}
    \hfill
     \begin{subfigure}[b]{0.3\textwidth}
         \centering
         \includegraphics[width=\textwidth]{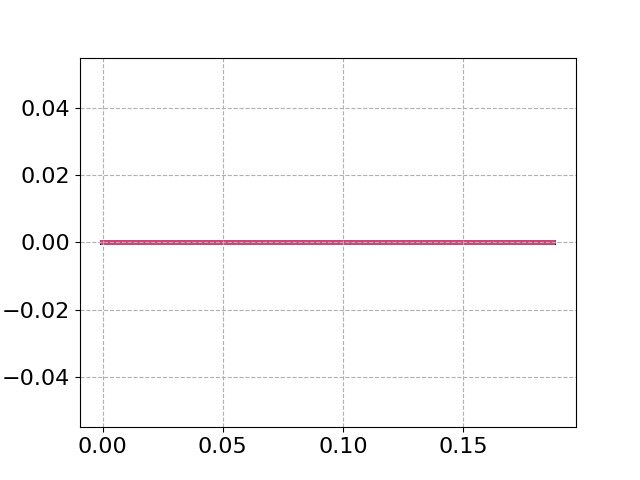}
         \caption{$\gamma=2$ (Bergman space)}
         \label{fig:ex_3.2_g=2_fs=16_plasma_nc_nasp.png}
     \end{subfigure}
    \caption{Comparison of Berezin ranges of the projection operator $P(f)=\langle f,z^2 \rangle 3z^2$ on $\mathcal{H}_\gamma(\mathbb{D})$ for different values of $\gamma.$}
    \label{fig:Example_3.2}
\end{figure}

Just we have seen that the Berezin range of a finie-rank operator of the form $T(f)=\langle f,g\rangle h,$ where $f,g,h\in \mathcal{H}_\gamma(\mathbb{D})$ is not always convex. However, if $f$ and $g$ both are some positive integral power of $z$, then its Berezin range (which is either a disc or an interval, depending on the powers) is always convex; we will prove this in the following theorem, which can also be regarded as a generalization of Theorem \ref{th:same integral power}. 
\begin{theorem}\label{th:different integral power}
    Let $T(f)=\langle f,z^m\rangle z^n;\;m,n\in \mathbb{N},$ be a rank-one operator acting on $\mathcal{H}_\gamma(\mathbb{D}).$ Then the Berezin range of $T$ is given by
    \begin{equation*}
        \textnormal{Ber}(T)=\begin{cases}
            \mathbb{D}_{\left(\frac{2\gamma}{m+n+2\gamma}\right)^\gamma \left(\frac{m+n}{m+n+2\gamma}\right)^{\frac{m+n}{2}}},&\text{when}\;m\neq n \\
            \left[0,\;\dfrac{\gamma^\gamma n^n}{(n+\gamma)^{n+\gamma}}\right],&\text{when}\;m=n
        \end{cases};
    \end{equation*}
    in either case, which is convex in the complex plane $\mathbb{C}.$
\end{theorem}
 \setcounter{case}{0}
\begin{proof}
    For $\lambda\in \mathbb{D},$ we have
    \begin{equation*}
        T\left(k_\lambda\right)=\langle k_\lambda,z^m\rangle z^n=\overline{\langle z^m,k_\lambda\rangle}z^n=\overline{\lambda}^mz^n.
    \end{equation*}
    \begin{case}\label{1st_case}
        m<n.
    \end{case}
   In this case, the Berezin transform of $T$ at $\lambda$ is given by
   \begin{equation*}
       \begin{split}
           \widetilde{T}(\lambda)&=\langle T\widehat{k}_\lambda,\widehat{k}_\lambda\rangle\\
           &=\dfrac{1}{\|k_\lambda\|^2}\langle Tk_\lambda,k_\lambda\rangle \\
           &=\left(1-|\lambda|^2\right)^\gamma \left(Tk_\lambda\right)(\lambda)\\
           &=\left(1-|\lambda|^2\right)^\gamma|\lambda|^{2m}\lambda^{n-m}.
       \end{split}
   \end{equation*}
   \par Put $\lambda=re^{i\theta}\;(0\leq r<1)$ to have
   \begin{equation*}
       \widetilde{T}(\lambda)=\left(1-r^2\right)^\gamma r^{m+n}e^{i\theta (n-m)}
   \end{equation*}
   \par For each fixed $r\in [0,1),$ the set $\left(1-r^2\right)^\gamma r^{m+n}e^{i\theta (n-m)},\theta \in \mathbb{R},$ describes a 
 circle in the complex plane $\mathbb{C}$ whose center is at the origin and radius is less than $1$ [notice that $\left(1-r^2\right)^\gamma r^{m+n}<1$ as $r<1$]. Conversely, let $\rho e^{it},\;0\leq \rho <1,\;t\in [0,2\pi),$ be any circle having center at the origin and of radius $<1.$ Since $0\leq \rho<1,$ it is always possible to find $r\in [0,1)$ to satisfy $\rho=\left(1-r^2\right)^\gamma r^{m+n}$ so that
 \begin{equation*}
     \rho e^{it}=\left(1-r^2\right)^\gamma r^{m+n}e^{it}=\left(1-r^2\right)^\gamma r^{m+n}e^{i\left(\frac{t}{n-m}\right)(n-m)}.
 \end{equation*}
 Thus, every circle having center at the origin and of radius $<1$ also corresponds to a Berezin transform $\widetilde{T}(\lambda)$ for some $\lambda\in \mathbb{D}.$ This means that the members in the Berezin range of $T$ constitute a system of concentric circles having center at the origin. Also, since $0\in \textnormal{Ber}(T)$ [notice that $\widetilde{T}(0)=0$], $\textnormal{Ber}(T)$ is a circular disc centered on the origin, which, clearly is a convex set in $\mathbb{C}.$ Now, to find the radius of this disc, we determine the extreme points of $\left(1-r^2\right)^\gamma r^{m+n}=F(r)\;\mbox{(say)}$; and to do this we differentiate it with respect to $r$ and equate to zero to have
 \begin{equation*}
     \left(1-r^2\right)^{\gamma-1}r^{m+n-1}\left\{m+n-(m+n+2\gamma)r^2\right\}=0.
 \end{equation*}
 Since $1-r^2\neq 0,$ above holds if and only if $r=0$ or $r=\sqrt{\dfrac{m+n}{m+n+2\gamma}}.$ We observe that when $r=0,\;F(r)=0;$ and when $r=\sqrt{\dfrac{m+n}{m+n+2\gamma}},\;F(r)=\left(\dfrac{2\gamma}{m+n+2\gamma}\right)^\gamma \left(\dfrac{m+n}{m+n+2\gamma}\right)^{\dfrac{m+n}{2}}.$ So, the radius of the aforesaid disc is $\left(\dfrac{2\gamma}{m+n+2\gamma}\right)^\gamma \left(\dfrac{m+n}{m+n+2\gamma}\right)^{\dfrac{m+n}{2}}.$
 \begin{case}\label{2nd_case}
     m>n.
 \end{case}
  In this case, the Berezin transform of $T$ at $\lambda$ is given by
   \begin{equation*}
       \begin{split}
           \widetilde{T}(\lambda)&=\langle T\widehat{k}_\lambda,\widehat{k}_\lambda\rangle\\
           &=\dfrac{1}{\|k_\lambda\|^2}\langle Tk_\lambda,k_\lambda\rangle \\
           &=\left(1-|\lambda|^2\right)^\gamma \left(Tk_\lambda\right)(\lambda)\\
           &=\left(1-|\lambda|^2\right)^\gamma|\lambda|^{2n}\lambda^{m-n}.
       \end{split}
   \end{equation*}
   \par Proceeding as in \textbf{Case \ref{1st_case}}, we find that the Berezin range of $T,\;\textnormal{Ber}(T)$, in this case also, is the disc centered on the origin and of radius $\left(\dfrac{2\gamma}{m+n+2\gamma}\right)^\gamma \left(\dfrac{m+n}{m+n+2\gamma}\right)^{\dfrac{m+n}{2}}.$
   \begin{case}\label{3rd_case}
       m=n.
   \end{case}
   In this case, the Berezin transform of $T$ at $\lambda$ is given by
   \begin{equation*}
       \widetilde{T}(\lambda)
       =\langle T\widehat{k}_\lambda,\widehat{k}_\lambda\rangle
       =\left(1-|\lambda|^2\right)^\gamma \left(Tk_\lambda\right)(\lambda)
       =\left(1-|\lambda|^2\right)^\gamma|\lambda|^{2n}.
    \end{equation*}
    where $|\lambda|\in[0,1).$ As shown earlier [c.f. proof of Theorem \ref{th:same integral power}], the Berezin range of $T,$ in this case, is the interval $\left[0,\;\dfrac{\gamma^\gamma n^n}{\left(n+\gamma\right)^{n+\gamma}}\right]$ on the real axis, which is a convex set in $\mathbb{C}.$
\end{proof}
\begin{remark}
    Since the result derived in Theorem \ref{th:different integral power} is symmetric with respect to $m$ and $n,$ we see that the operators $T_1f=\langle f,z^m\rangle z^n$ and $T_2f=\langle f,z^n\rangle z^m$ acting on the weighted Hardy space $\mathcal{H}_\gamma(\mathbb{D})$ have the same Berezin range.
\end{remark}
\begin{remark}
    Putting $\gamma=1$ in the result of Theorem \ref{th:different integral power} we see that the Berezin range of the rank-one operator $Tf=\langle f,z^m\rangle z^n$ acting on the classical Hardy space $H^2(\mathbb{D})$ is the disc $\D_{\left(\frac{2}{m+n+2}\right) \left(\frac{m+n}{m+n+2}\right)^{\frac{m+n}{2}}}$ or the interval $\left[0,\;\dfrac{n^n}{\left(n+1\right)^{n+1}}\right]$ according as $m\neq n$ or $m=n$, which have been proved in \cite{Athul2_Article_2024}. \par
    Also, putting $\g=2$ in the above result we see that the Berezin range of the rank one operator $Tf=\langle f,z^m\rangle z^n$ acting on the Bergman space $A^2(\mathbb{D})$ is the disc $\D_{\left(\frac{4}{m+n+4}\right)^2 \left(\frac{m+n}{m+n+4}\right)^{\frac{m+n}{2}}}$ or the interval $\left[0,\;\dfrac{4 n^n}{\left(n+2\right)^{n+2}}\right]$ according as $m\neq n$ or $m=n$, which have been proved in 
 \cite{Athul2_Article_2024}. 
\end{remark}
\begin{corollary}
    The Berezin range of the rank-one operator $Tf=\langle f, z^m\rangle z^n,\;m,n\in \mathbb{N}$ acting on the space $\mathcal{H}_\gamma(\mathbb{D})$ is symmetric about the real axis.
\end{corollary}
\begin{proof}
    This follows directly from Theorem \ref{th:different integral power}.
\end{proof}
\begin{corollary}
    Let $Tf=\langle f, z^m\rangle z^n,\;m,n\in \mathbb{N}$ be a rank-one operator acting on the space $\mathcal{H}_\gamma(\mathbb{D})$. then for every $\lambda\in \mathbb{D},\;\mathcal{R}\left\{\widetilde{T}(\lambda)\right\}\in \textnormal{Ber}(T).$
\end{corollary}
\begin{proof}
    Let $\lambda\in \mathbb{D}$ be arbitrary. Since $\textnormal{Ber}(T)$ is symmetric about the real axis, $\overline{\widetilde{T}(\lambda)}\in \textnormal{Ber}(T).$ Then, since $\textnormal{Ber}(T)$ is convex, $\dfrac{1}{2}\widetilde{T}(\lambda)+\dfrac{1}{2}\overline{\widetilde{T}(\lambda)}=\mathcal{R}\left\{\widetilde{T}(\lambda)\right\}\in \textnormal{Ber}(T).$
\end{proof}
\begin{example}
    Let $T(f)=\langle f,z^2\rangle z^3$ be a rank-one operator acting on the space $\mathcal{H}_{\frac{1}{2}}(\mathbb{D})$ (putting $\gamma=\dfrac{1}{2}$). Then, $\textnormal{Ber}(T)=\D_{\frac{25\sqrt{5}}{216}}$, which is convex in $\mathbb{C}$ (see Figure \ref{fig:th_7_g=0.5_m=2,n=3_fs=16_wc.png}).
\end{example}
\begin{figure}[ht!]
    \centering
    \includegraphics[width=0.5\textwidth]{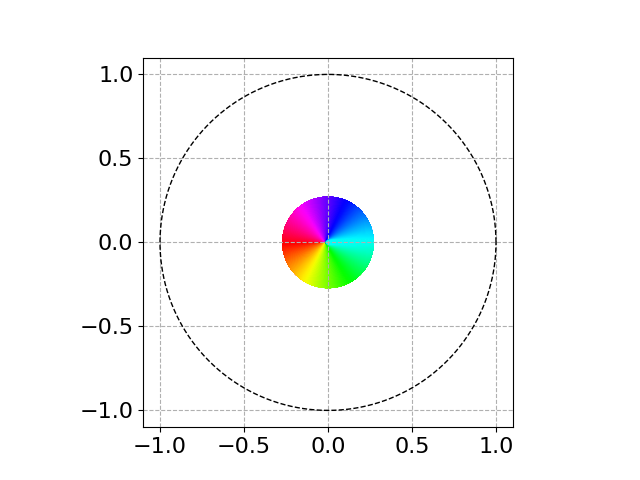}
    \caption{Berezin range of the operator $T(f)=\langle f,z^2 \rangle z^3$ on $\mathcal{H}_{1/2}(\mathbb{D}).$}
    \label{fig:th_7_g=0.5_m=2,n=3_fs=16_wc.png}
\end{figure}
\begin{remark}
    Let $T(f)=\langle f, z^m \rangle z^n$ be a rank-one operator acting on the space $\mathcal{H}_\gamma(\mathbb{D}).$ Keeping $m,n$ fixed and letting $\gamma \rightarrow \infty$, we see that $\left(\dfrac{2\gamma}{m+n+2\gamma}\right)^\gamma \left(\dfrac{m+n}{m+n+2\gamma}\right)^{\dfrac{m+n}{2}} \rightarrow 0$ and $\dfrac{\gamma^\gamma n^n}{(n+\gamma)^{n+\gamma}}\rightarrow 0.$ This indicates that the Berezin range of $T,$ which is either a disc (when $m\neq n$ or an interval on the real axis (when $m=n$), gradually squeezes with increasing $\gamma;$ and finally, when $\gamma \rightarrow \infty,\;\textnormal{Ber}(T)$ tends to coincide with the origin. This can also be verified with the aid of computer (see Figure \ref{fig:comparison3}).
\end{remark}

\begin{figure}[htbp!]
    \centering
     \begin{subfigure}[b]{0.3\textwidth}
         \centering
         \includegraphics[width=\textwidth]{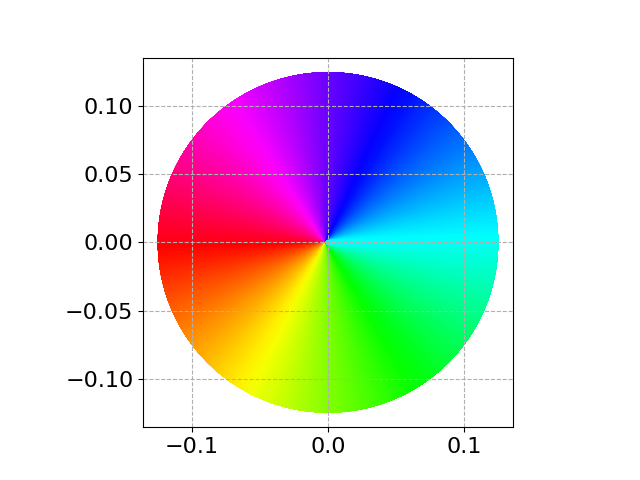}
         \caption{$\gamma=1$ (Hardy space)}
         \label{fig:th_7_g=1_m=2,n=3_fs=16_nc.png}
     \end{subfigure}
    \hfill
     \begin{subfigure}[b]{0.3\textwidth}
         \centering
         \includegraphics[width=\textwidth]{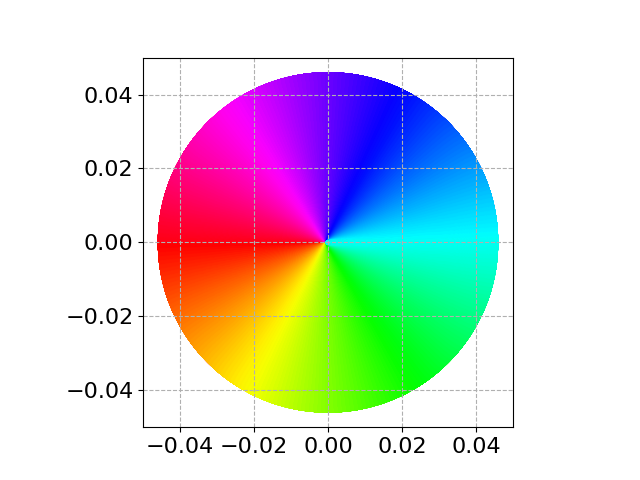}
         \caption{$\gamma=2$ (Bergman space)}
         \label{fig:th_7_g=2_m=2,n=3_fs=16_nc.png}
     \end{subfigure}
    \hfill
     \begin{subfigure}[b]{0.3\textwidth}
         \centering
         \includegraphics[width=\textwidth]{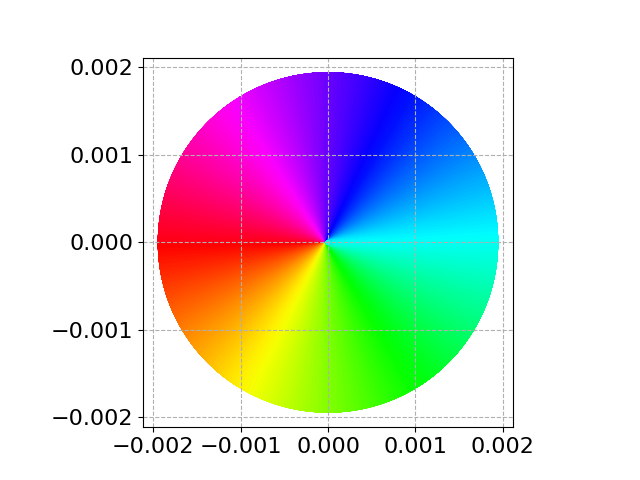}
         \caption{$\gamma=10$}
         \label{fig:th_7_g=10_m=2,n=3_fs=16_nc.png}
     \end{subfigure}
    \caption{Comparison of Berezin ranges of the operator $T(f)=\langle f,z^2 \rangle z^3$ acting on the weighted Hardy space $\mathcal{H}_\gamma(\mathbb{D})$ with different weights.}
     \label{fig:comparison3}
\end{figure}

\medskip
\subsection{Berezin range of multiplication operator}
For the space $\mathcal{H}_\gamma (\mathbb{D}),$ let
\begin{equation*}
    \text{Mult}\left(\mathcal{H}_\gamma(\mathbb{D})\right)=\left\{\varphi \in \mathcal{H}_\gamma (\mathbb{D}):\varphi f \in \mathcal{H}_\gamma (\mathbb{D})\;\text{for all}\;f \in \mathcal{H}_\gamma(\mathbb{D})\right\}.
\end{equation*}
For $\varphi \in \text{Mult}\left(\mathcal{H}_\gamma(\mathbb{D})\right),$ define the multiplication operator $M_\varphi$ acting on $\mathcal{H}_\gamma(\mathbb{D})$ by
\begin{equation}
    (M_\varphi f)(z)=\phi(z)f(z),\;f\in \mathcal{H}_\gamma(\mathbb{D}).
\end{equation}
In particular, if $\varphi(z)=a_0+a_1z+a_2z^2+\cdots+a_nz^n,\;a_i\in\mathbb{C}\;\text{with}\;a_n \neq 0,\;i=1(n)$ be a complex polynomial, then the operator $M_\varphi$ is of finite rank. In the following, we investigate the convexity of the Berezin range of the operator $M_\varphi$.
\begin{theorem}\label{th:berezin range_mult}
    Let $\varphi \in \textnormal{Mult}\left(\mathcal{H}_\gamma(\mathbb{D}).\right)$ Then the Berezin range of the multiplication operator $M_\varphi$ acting on the weighted Hardy space $\mathcal{H}_\gamma(\mathbb{D})$ is $\varphi(\mathbb{D}),$ the image of the unit disc $\mathbb{D}$ under the mapping $\varphi.$
\end{theorem}
\begin{proof}
    For $\lambda \in \mathbb{D},$ we have
    \begin{equation*}
        M_\varphi(k_\lambda)=\varphi(z)k_\lambda(z).
    \end{equation*}
    Then the Berezin transform of $M_\varphi$ at $\lambda\in \mathbb{D})$ is given by
    \begin{align*}
        \widetilde{M_\varphi}(\lambda)=\langle M_\varphi \widehat{k}_\lambda,\widehat{k}_\lambda \rangle 
        =\dfrac{1}{\|k_\lambda\|^2}\langle M_\varphi k_\lambda,k_\lambda \rangle 
        =\dfrac{1}{\|k_\lambda\|^2} \left(M_\varphi k_\lambda\right)(\lambda)
        =\dfrac{1}{\|k_\lambda\|^2} \varphi(\lambda)k_\lambda(\lambda)\\
        =\dfrac{1}{\|k_\lambda\|^2} \varphi(\lambda) \langle 
        k_\lambda,k_\lambda \rangle 
        =\dfrac{1}{\|k_\lambda\|^2} \varphi(\lambda)\|k_\lambda\|^2
        =\varphi(\lambda).
    \end{align*}
    Hence, the Berezin range of $M_\varphi$ is given by $\textnormal{Ber}\left(M_\varphi\right)=\left\{\widetilde{M_\varphi}(\lambda):\lambda \in \mathbb{D}\right\}=\left\{\varphi(\lambda):\lambda \in \mathbb{D}\right\}=\varphi(\mathbb{D}).$ 
\end{proof}
\begin{remark}
        The absence of any term involving $\gamma$ in the result of Theorem \ref{th:berezin range_mult} indicates that the result is independent of the space $\mathcal{H}_\gamma (\mathbb{D}).$ Indeed, Cowen et al. in \cite[Proposition 3.2]{CowenFelder} have derived the same result in any RKHS.
    \end{remark}
\begin{corollary}\label{cor:mult_z^n}
    The Berezin range of the multiplication operator $M_{z^n},\;n\in \mathbb{N},$ acting on $\mathcal{H}_\gamma (\mathbb{D})$ is the unit disc $\mathbb{D},$ which is convex in $\mathbb{C}.$
\end{corollary}
\begin{proof}
    By Theorem \ref{th:berezin range_mult}, the Berezin range of $M_{z^n},\;\textnormal{Ber}(M_{z^n})$ is the image of the unit disc $\mathbb{D}$ under the mapping $w=z^n,\;z\in \mathbb{D}.$ We see that for $z=0,\;w=0$. For $z\neq0$, we put $z=re^{i\theta}\;(0<r<1)$ to have
    \begin{equation*}
        w=r^ne^{in\theta},
    \end{equation*}
  and we see from above equation that the transformation $w=z^n$ contracts the radius
vector representing $z$ and rotates it through the angle $(n-1)\theta$ about the origin. The image of $\mathbb{D}$ is, therefore, geometrically similar to itself, i.e., a disc centered on the origin. Now, to find the radius of this image disc, we notice that $r^n$ increases with increasing $r;$ and when $r\rightarrow 1^{-},\;r^n \rightarrow 1^{-}.$ Thus, the radius of the image disc is $1,$ i.e., it is nothing but the disc $\mathbb{D}$ itself.
\end{proof}
\begin{corollary}\label{cor:mult_Az^n}
    The Berezin range of the multiplication operator $M_{Az^n},\;A \in \mathbb{C},\;n\in \mathbb{N},$ acting on $\mathcal{H}_\gamma (\mathbb{D})$ is the open  disc $|z|<|A|$ in the complex plane, which is convex in $\mathbb{C}.$
\end{corollary}
\begin{proof}
    By Theorem \ref{th:berezin range_mult}, the Berezin range of $M_{Az^n},\;\textnormal{Ber}\left(M_{Az^n}\right)$ is the image of the unit disc $\mathbb{D}$ under the mapping $w=Az^n,\;z\in \mathbb{D}.$ We see that for $z=0,\;w=0$. For $z\neq0$, we put $z=re^{i\theta}\;(0<r<1)$ and $A=ae^{i\alpha}$ to have
    \begin{equation*}        w=ae^{i\alpha}r^ne^{in\theta}=ar^ne^{i(\alpha+n\theta)}
    \end{equation*}
  and we see from above equation that the transformation $w=Az^n$ expands or contracts (according as the modulus of $A$) the radius
vector representing $z$ and rotates it through the angle $[\alpha+(n-1)\theta]$ about the origin. The image of $\mathbb{D}$ is, therefore, geometrically similar to itself, i.e., a disc centered on the origin. Now, to find the radius of this image disc, we notice that $ar^n$ increases with increasing $r;$ and when $r\rightarrow 1^{-},\;ar^n \rightarrow a^{-}.$ Thus, the radius of the image disc is $a,$ i.e., $\textnormal{Ber}\left(M_{Az^n}\right)=\mathbb{D}_{|A|}$.
\end{proof}
For a visual interpretation of Corollary \ref{cor:mult_Az^n}, see Figure \ref{fig:mult_pol__1-i_z_3_fs=16.png}.
\begin{corollary}\label{cor:mult_Az^n+B}
    The Berezin range of the multiplication operator $M_{Az^n+B},\;A,B\in \mathbb{C},\;n\in \mathbb{N},$ acting on $\mathcal{H}_\gamma (\mathbb{D})$ is the open disc centered on the point $B$ and of radius $|A|,$ which is convex in $\mathbb{C}.$
\end{corollary}
\begin{proof}
    By Theorem \ref{th:berezin range_mult}, the Berezin range of $M_{Az^n+B},\;\textnormal{Ber}(M_{Az^n+B})$ is the image of the unit disc $\mathbb{D}$ under the mapping $w=Az^n+B,\;z\in \mathbb{D}.$ We can view this mapping as the composition of the mappings
    \begin{equation*}
        w=W+B,\hspace{0.3cm}W=Az^n;\hspace{0.3cm}z\in \mathbb{D}.
    \end{equation*}
    \par As we have seen earlier (c.f. Corollary \ref{cor:mult_Az^n}), under the mapping $Az^n,$ the unit disc $\mathbb{D}$ is mapped onto the open disc $\mathbb{D}_{a}$ in the complex plane, where $a=|A|$. Also, the mapping $w=W+B$ is a translation by means of the vector representing $B$ (see \cite{Brown_Churchill_Book_2009}, p. 312), and hence the image region is congruent to the original one. Thus, if $B=b_1+b_2i,$ then the image of $\mathbb{D}$ under the mapping $w=Az^n+B$ is the disc $|z-(b_1,b_2)|<a$ in the complex plane, which is the Berezin range of $M_{Az^n+B}.$  
\end{proof}
For a visual interpretation of Corollary \ref{cor:mult_Az^n+B}, (see Figure \ref{fig:mult_pol_3iz_4+2+i_fs=16.png}). \par
\vspace{0.3cm}
Consider the multiplication operator $M_\varphi$ induced by the complex polynomial $\varphi(z)=a_0+a_1z+a_2z^2+\cdots+a_nz^n,$ where $a_i\in \mathbb{C},\;i=1(n),$ defined on $\mathbb{D}.$ Then the Berezin range of the operator $M_\varphi$ acting on the weighted Hardy space $\mathcal{H}_\gamma (\mathbb{D})$ is, by Theorem \ref{th:berezin range_mult}, the image of $\mathbb{D}$ under the mapping $w=\varphi(z).$ Since the image of a disc under a polynomial mapping is not necessarily convex in $\mathbb{C},$ it follows that $\textnormal{Ber}(M_\varphi)$ is not always convex in $\mathbb{C}$ (see Figure \ref{fig:mult_pol_3iz_4+5z-2i_fs=16.png}).

\begin{figure}[htbp!]
    \centering
    \begin{subfigure}[b]{0.3\textwidth}
        \centering
        \includegraphics[width=\textwidth]{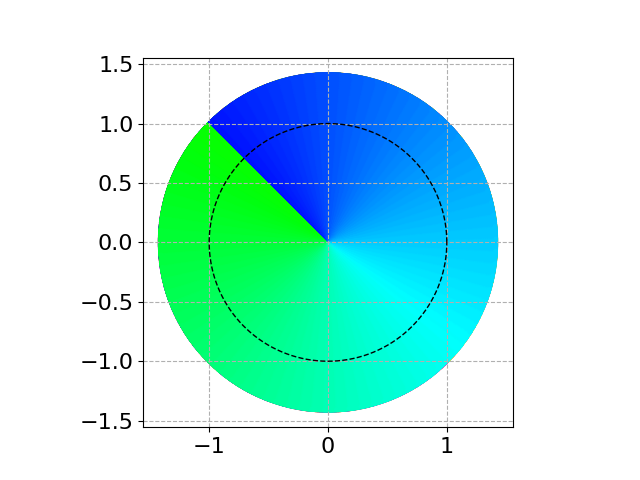}
        \caption{\centering $\phi(z)=(1-i)z^3$(Apparently convex)}
        \label{fig:mult_pol__1-i_z_3_fs=16.png}
    \end{subfigure}
    \hfill
    \begin{subfigure}[b]{0.3\textwidth}
        \centering
        \includegraphics[width=\textwidth]{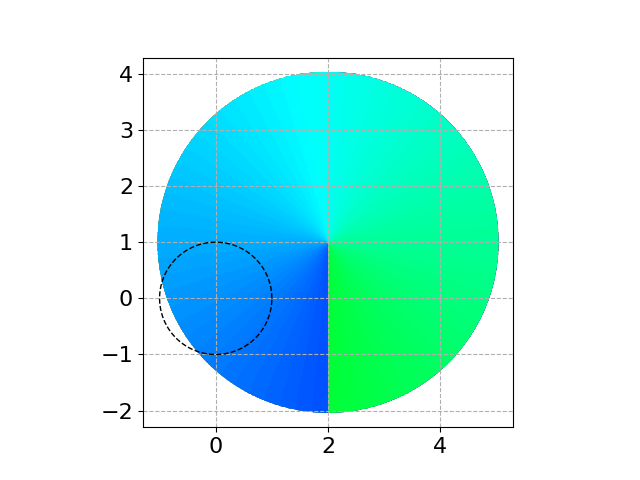}
        \caption{\centering $\phi(z)=3iz^4+2+i$(Apparently convex)}
        \label{fig:mult_pol_3iz_4+2+i_fs=16.png}
    \end{subfigure}
    \hfill
    \begin{subfigure}[b]{0.3\textwidth}
        \centering
        \includegraphics[width=\textwidth]{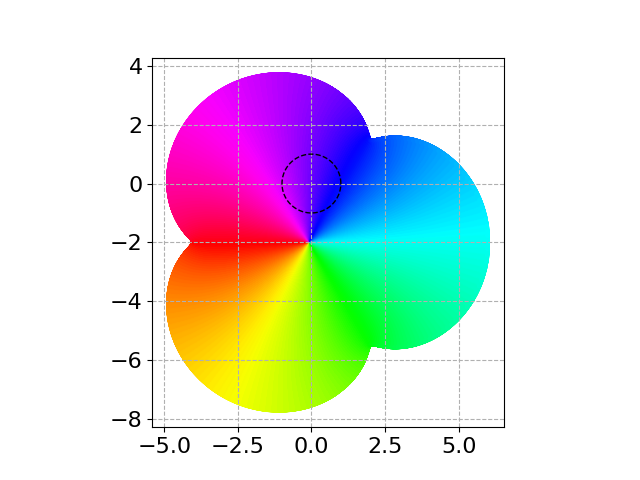}
        \caption{\centering $\phi(z)=z^4+5z-2i$(Apparently not convex)}
        \label{fig:mult_pol_3iz_4+5z-2i_fs=16.png}
    \end{subfigure}
    \caption{Comparison of Berezin ranges of the multiplication operator $M_\varphi$ acing on $\mathcal{H}_\gamma (\mathbb{D}),$ for different polynomials $\varphi(z)$.}
    \label{fig:mult_pol}
\end{figure}
\begin{remark}
    Even when all the coefficients of $\varphi(z)$ are purely real, there is no guarantee that $\textnormal{Ber}\left(M_\varphi\right)$ is convex (see Figure \ref{fig:mult_pol_z_14+5z_2-2z+3_fs=16.png}). Only one thing that can be asserted about $\textnormal{Ber}\left(M_\varphi\right)$ is as follows.
\end{remark}
\begin{corollary}\label{cor:sym_pol}
    If $\varphi(z)=a_0+a_1z+a_2z^2+\cdots+a_nz^n$ be a polynomial with all coefficients purely real, then $\textnormal{Ber}\left(M_\varphi\right)$ is closed with respect to complex conjugation and hence is symmetric about the real axis.
\end{corollary}
\begin{proof}
    This follows from the facts that for all $z \in \mathbb{D},\;\overline{z}\in \mathbb{D}$ and $\overline{\varphi(z)}=\varphi(\overline{z}),$ when all $a_i\in \mathbb{R}.$ 
\end{proof}
For a visual verification of Corollary \ref{cor:sym_pol}, see Figure \ref{fig:mult_pol_z_14+5z_2-2z+3_fs=16.png}.
\begin{figure}[ht!]
    \centering
    \includegraphics[width=0.5\textwidth]{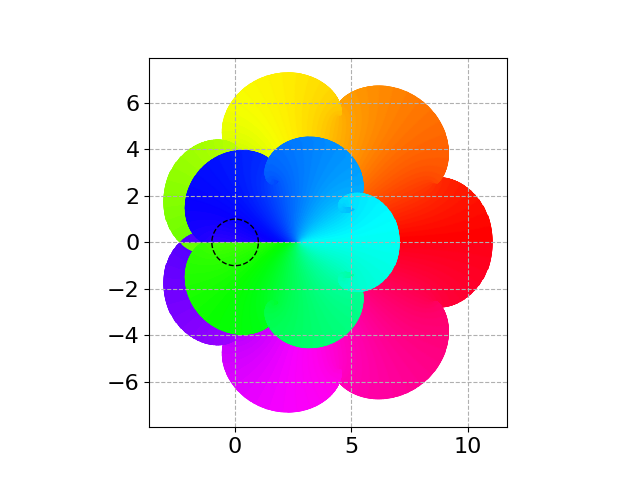}
    \caption{Berezin range of $M_\varphi$, where $\varphi(z)=z^{14}+5z^2-2z+3$ (apparently not convex).}
    \label{fig:mult_pol_z_14+5z_2-2z+3_fs=16.png}
\end{figure}
\\Our next goal is to investigate the convexity of the Berezin ranges of multiplication operators induced by a finite Blaschke product. First, we define the \textit{finite Blascke product.}
\begin{definition}[Finite Blaschke Product]\label{def:fbp}
    Let $\alpha_1,\alpha_2,\dots,\alpha_n$ be a finite set of points (not necessarily all distinct) inside $\mathbb{D}.$ Then a finite Blaschke product of degree $n$ is defined as
    \begin{equation}\label{eq:fbp}        B(z)=\zeta z^m\prod_{k=1}^n\dfrac{|\alpha_k|}{\alpha_k}\dfrac{\alpha_k-z}{1-\overline{\alpha_k}z},
    \end{equation}
    where $|\zeta|=1$ and $m$ is a non-negative integer, provided none of $\alpha_k\neq0$. If, however, some $\alpha_k\neq0,$ we set $\dfrac{|\alpha_k|}{\alpha_k}=-1.$ \textnormal{For more details on finite Blaschke product, one can see} \cite{Garcia_book}.
\end{definition}

\begin{theorem}\label{th:mult_fbp}
    The Berezin range of the multiplication operator $M_{B(z)}$ acting on $\mathcal{H}_\gamma(\mathbb{D}),$ where $B(z)$ is the finite Blaschke product of degree $n$ given by \eqref{eq:fbp}, is the unit disc $\mathbb{D}$ itself, which is convex in $\mathbb{C}.$
\end{theorem}
\begin{proof}
    By Theorem \ref{th:berezin range_mult}, $\textnormal{Ber}\left(M_{B(z)}\right)=B(\mathbb{D}).$ First, we claim that $B(z)=1$ for all $z\in \mathbb{T},$ where $\mathbb{T}$ denotes the unit circle, i.e., the boundary of $\mathbb{D}.$ We note that for any $z\in \mathbb{T},\;z\overline{z}=|z|^2=1$ so that for any $z\in \mathbb{T},$
    \begin{equation*}
        \begin{split}
            |B(z)|=|\zeta||z^m|\left|\prod_{k=1}^n\dfrac{|\alpha_k|}{\alpha_k}\dfrac{\alpha_k-z}{1-\overline{\alpha_k}z}\right|    &=\prod_{k=1}^n\left|\dfrac{|\alpha_k|}{\alpha_k}\right|\left|\dfrac{\alpha_k-z}{1-\overline{\alpha_k}z}\right|\\     &=\prod_{k=1}^n\left|\dfrac{\alpha_k-z}{z\left(\overline{z}-\overline{\alpha_k}\right)}\right|\\
            &=\left|\dfrac{1}{z}\right|\prod_{k=1}^n\dfrac{|z-\alpha_k|}{\left|\overline{z-\alpha_k}\right|}\\
            &=1.
        \end{split}
    \end{equation*}\par
    Thus, our claim is established. Now, since $B(z)$ is analytic in $\mathbb{D}$ and continuous on $\overline{\mathbb{D}}$, by \textit{maximum modulus principle} (for instance, see \cite[p. 178]{Brown_Churchill_Book_2009}) $B(z)$ attains its maximum value on $\overline{\mathbb{D}}$ and not in $\mathbb{D}.$ Then, for all $z\in \mathbb{D},\;|B(z)|<1.$ Thus, for $z\in \mathbb{D},\;B(z)\in \mathbb{D}.$ Conversely, let $w \in \mathbb{D}.$ Then the equation $B(z)=w$ has exactly $n$ solution all of which are in $\mathbb{D}$ (see \cite[Theorem 3.6]{Garcia_Survey}). This means that every $w\in \mathbb{D}$ has at least one pre-image in $\mathbb{D}$ under the mapping $B.$ Hence, $B(\mathbb{D})=\mathbb{D}.$
\end{proof}
\vspace{0.3cm}
Let us now consider a disc automorphism (also known as \textit{Blaschke factor}) given by
\begin{equation}\label{eq:bf}
    \phi(z)=\zeta \dfrac{\alpha-z}{1-\overline{\alpha}z},
\end{equation}
where $\alpha \in \mathbb{D},\,\zeta \in \mathbb{D}.$ Since a disk automorphism is a particular type of finite Blaschke product (with one zero only), we have the following corollary.
\begin{corollary}
    The Berezin range of the multiplication operator $M_\phi$ acting on $\mathcal{H}_\gamma (\mathbb{D}),$ where $\phi$ is the Blascke factor given by (\ref{eq:bf}) is the unit disc $\mathbb{D}.$
\end{corollary}
If, however, we let $\zeta$ to lie anywhere in the complex plane, then proceeding as in the proof of Theorem \ref{th:mult_fbp} we can see that the Berezin range of $M_\phi$ is a disc centered on the origin and radius $|\zeta|$. For a visual verification, see Figure \ref{fig:mult_bf_different zeta}.

\begin{figure}[htbp!]
    \centering
    \begin{subfigure}[b]{0.3\textwidth}
        \centering
        \includegraphics[width=\textwidth]{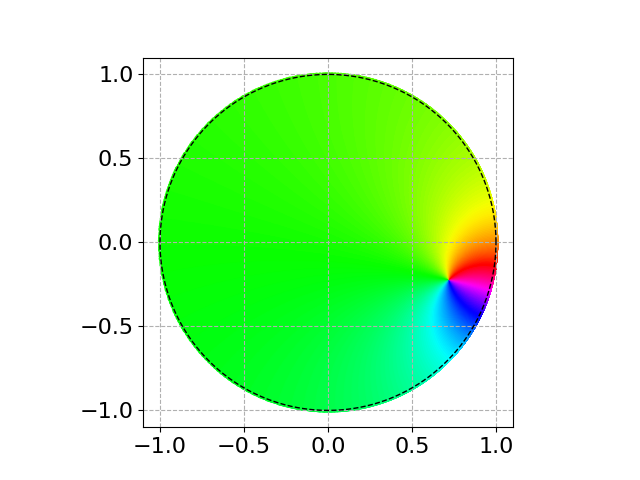}
        \caption{$\zeta=\frac{1}{\sqrt{2}}(1+i);\;|\zeta|=1$}
        \label{fig:mult_bf_l=_1+i_by_sqrt_2___a=_1-2i_by_3__fs=16.png}
    \end{subfigure}
    \hfill
    \begin{subfigure}[b]{0.3\textwidth}
        \centering
        \includegraphics[width=\textwidth]{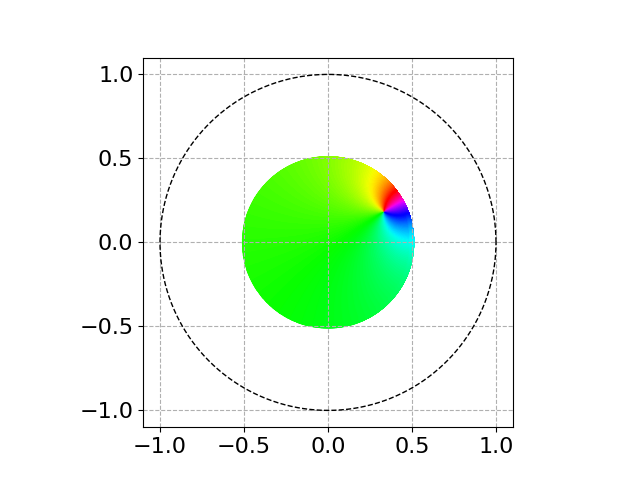}
        \caption{$\zeta=\frac{i}{2};\;|\zeta|=0.5<1$}
        \label{fig:mult_bf_l=_i_by_2__a=_1-2i_by_3__fs=16.png}
    \end{subfigure}
    \hfill
    \begin{subfigure}[b]{0.3\textwidth}
        \centering
        \includegraphics[width=\textwidth]{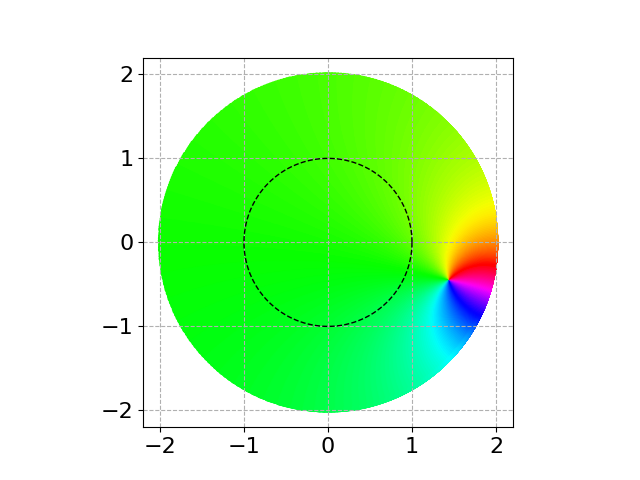}
        \caption{$\zeta=\sqrt{2}(1+i);\;|\zeta|=2>1$}
        \label{fig:mult_bf_l=sqrt_2__1+i__a=_1-2i_by_3__fs=16.png}
    \end{subfigure}
    \caption{Comparison of Berezin ranges of $M_\phi$, where $\phi(z)=\zeta \dfrac{\alpha-z}{1-\overline{\alpha}z},\;\alpha=\frac{1}{3}(1-2i),$ for different values of $\zeta$.}
    \label{fig:mult_bf_different zeta}
\end{figure}

\medskip
In the following, two tables comparing the Berezin ranges respectively of rank-one operators of different form acting on different spaces (see Table \ref{tab:comparison}) and of multiplication operators induced by different analytic functions (see Table  \ref{tab:comparison_mult}) are given, which may be helpful to the readers to grasp the summary of this article.

\begin{table}[htbp]
    \centering
    \begin{tabular}{|c||c|c|c|}
    \hline
    \makecell{\textbf{Hilbert spaces}$\;\;\Rightarrow$ \\ \textbf{Operators}$\;\;\Downarrow$} & \makecell{\textbf{Hardy space} \\ $H^2(\mathbb{D})$} & \makecell{\textbf{Bergman space}\\ $A^2(\mathbb{D})$} & \makecell{\textbf{Weighted Hardy space}\\ $\mathcal{H}_\gamma(\mathbb{D})$} \\
    \hline\hline
    $T(f)=\langle f,z\rangle z$ & $\left[0,\;\frac{1}{4}\right]$ & $\left[0,\;\frac{4}{27}\right]$ & $\left[0,\;\dfrac{\gamma^\gamma}{\left(1+\gamma\right)^{1+\gamma}}\right]$ \\
    \hline
    \makecell{$T(f)=\langle f,z^n\rangle z^n,$ \\where $n\in \mathbb{N}$} & $\left[0,\;\dfrac{n^n}{(n+1)^{n+1}}\right]$ & $\left[0,\;\dfrac{4n^n}{(n+2)^{n+2}}\right]$ & $\left[0,\;\dfrac{\gamma^\gamma n^n}{\left(n+\gamma\right)^{n+\gamma}}\right]$ \\
    \hline
    \makecell{$T(f)=\langle f,az^n\rangle az^n,$\\ where $n\in \mathbb{N},\;a\in \mathbb{C}$} & $\left[0,\;|a|^2\dfrac{n^n}{(n+1)^{n+1}}\right]$ & $\left[0,\;|a|^2\dfrac{4n^n}{(n+2)^{n+2}}\right]$ & $\left[0,\;|a|^2\dfrac{\gamma^\gamma n^n}{\left(n+\gamma\right)^{n+\gamma}}\right]$ \\
    \hline
   \makecell{$T(f)=\displaystyle\sum_{n=1}^{\infty}{\langle f,az^n\rangle az^n}$,\\ where $a\in \mathbb{D}$} & $[0,|a|^2)$ & $\left[0,\dfrac{|a|^2}{4}\right]$ & \makecell{$\left[0,\;|a|^2\dfrac{(\gamma-1)^{\gamma-1}}{\gamma^{\gamma}}\right],$ \\ when $\gamma>1;$\\$\left[0,\;|a|^2\right),$ \\when $\gamma=1;$ \\ $[0,\infty),$ \\when $0<\gamma\leq1.$} \\
    \hline
    \makecell{$T(f)=\langle f,z^m\rangle z^n,$ \\ where $m,n\in \mathbb{N}$} & \makecell{$\mathbb{D}_{\left(\frac{2}{m+n+2}\right) \left(\frac{m+n}{m+n+2}\right)^{\frac{m+n}{2}}},$ \\ when $m\neq n;$ \\ $\left[0,\;\dfrac{n^n}{\left(n+1\right)^{n+1}}\right],$\\ when $m=n.$} & \makecell{$\mathbb{D}_{\left(\frac{4}{m+n+4}\right)^2 \left(\frac{m+n}{m+n+4}\right)^{\frac{m+n}{2}}},$\\ when $m\neq n;$\\$\left[0,\;\dfrac{4 n^n}{\left(n+2\right)^{n+2}}\right],$ \\ when $m=n.$} & \makecell{$\mathbb{D}_{\left(\frac{2\gamma}{m+n+2\gamma}\right)^\gamma \left(\frac{m+n}{m+n+2\gamma}\right)^{\frac{m+n}{2}}},$ \\ when $m\neq n;$ \\ $\left[0,\;\dfrac{\gamma^\gamma n^n}{\left(n+\gamma\right)^{n+\gamma}}\right],$ \\ when $m=n.$} \\
    \hline
\end{tabular}
    \caption{Comparison of Berezin ranges of different rank-one operators acting on different spaces.}
    \label{tab:comparison}
\end{table}

\begin{table}[ht!]
    \centering
    \begin{tabular}{|l|r|}
    \hline
    \textbf{Operator} & \textbf{Berezin range} \\
    \hline \hline
    $M_{z^n};\;n\in \mathbb{N}$ & $\mathbb{D}$ \\
    \hline
    $M_{Az^n};\;A\in \mathbb{C},\;n\in \mathbb{N}$ &  $|z|<|A|$, i.e., $\mathbb{D}_{|A|}$ \\
    \hline
    $M_{Az^n+B};\;A,B\in \mathbb{C},\;n\in \mathbb{N}$ &  $|z-B|<|A|$ \\
    \hline
     $M_{B(z)},$ where $B(z)$ is finite Blaschke product of degree $n$ & $\mathbb{D}$   \\
     \hline
     $M_{\phi(z)},$ where $\phi(z)=\zeta \dfrac{\alpha-z}{1-\overline{\alpha}z};\;\alpha \in \mathbb{D},\;\zeta \in \mathbb{C}$ & $\mathbb{D}_\zeta$ \\
     \hline
    \end{tabular}
    \caption{Comparison of Berezin ranges of different multiplication operators acting on the space $\mathcal{H}_\gamma (\mathbb{D})$ with any weight $\gamma.$}
    \label{tab:comparison_mult}
\end{table}
We conclude by remarking that the convexity of the Berezin range for finite-rank operators acting on the weighted Hardy space \( \mathcal{H}_\gamma(\mathbb{D}) \) over the unit disc \( \mathbb{D} \) are investigated in this article. By employing similar analysis to different special classes of operators, one can explore the convexity of the Berezin range of operators on any reproducing kernel Hilbert space (RKHS) of holomorphic functions.  In light of this, we end by posing the following problem for further investigation:\\
\textbf{Can one characterize the convexity of the Berezin range for a finite-rank operator on suitable reproducing kernel Hilbert spaces ?}

\vspace{.1cm}
	\noindent
	{\small {\bf Acknowledgments.}\\
	 The second author is supported by the Institute Postdoctoral Fellowship, IIT Bhubaneswar (F. 15-12/2024-Acad/SBS5-PDF-01).
	}\\
\noindent \textbf{Declarations}
\begin{itemize}
\item {\bf{Availability of data and materials}}: Not applicable.
\item {\bf{Competing interests}}: The authors declare that they have no competing interests.
\item {\bf{Funding}}: Not applicable.
\item {\bf{Authors' contributions}}: The authors declare that they have contributed equally to this paper. All authors have read and approved this version.
\end{itemize}

\bigskip


\end{document}